\documentclass{article}
\usepackage[T2C]{fontenc}
\usepackage[cp1251]{inputenc}
\usepackage[russian,english]{babel}
\usepackage[tbtags]{amsmath}
\usepackage{amsfonts,amssymb,mathrsfs,amscd}
\usepackage[im]{sks2e-tr}

\ifx\proof\undefined
  \newenvironment{proof}[1][Proof]{\par\noindent\textbf{#1. }\ignorespaces}{\hfill$\square$\par\medskip}
\fi
\numberwithin{equation}{section}
\newtheorem{theorem}{Theorem}
\newtheorem{lemma}{Lemma}[section]

\newtheorem{definition}{Definition}
\newtheorem{remark}{Remark}

\usepackage{cases,comment}
\usepackage{graphicx,xcolor,float}
\usepackage{fancyhdr}
\usepackage{booktabs}
\usepackage{varwidth}
\usepackage{physics}
\usepackage{cite}
\usepackage{color}

\usepackage{graphicx}
\usepackage{subcaption}

\newtheorem{cond}{Condition}
\newtheorem{assumption}{Assumption}

\begin{document}

\title{Internal layer solutions and coefficient recovery in time-periodic reaction-diffusion-advection equations}
\author{Dmitrii Chaikovskii}
\address{MSU-BIT-SMBU Joint Research Center of Applied Mathematics, Shenzhen MSU-BIT University, Shenzhen 518172, People's Republic of China}
\email{dmitriich@smbu.edu.cn} 

\author{Ye Zhang}
\address{School of Mathematics and Statistics, Beijing Institute of Technology, Beijing, 100081, China}
\email{ye.zhang@smbu.edu.cn}
\author{Aleksei Liubavin*}\thanks{*Corresponding author}
\address{Department of Computational Mathematics and Cybernetics, MSU-BIT-SMBU Joint Research Center of Applied Mathematics, Shenzhen MSU-BIT University, Shenzhen 518172, People's Republic of China}
\email{la1992@mail.ru}

\date{DD.MM.YYYY} 
\udk{} 

\maketitle

\begin{fulltext}

\begin{abstract}
This article investigates the non-stationary reaction-diffusion- advection equation, emphasizing solutions with internal layers and the associated inverse problems. We examine a nonlinear singularly perturbed partial differential equation (PDE) within a bounded spatial domain and an infinite temporal domain, subject to periodic temporal boundary conditions. A periodic asymptotic solution featuring an inner transition layer is proposed, advancing the mathematical modeling of reaction-diffusion- advection dynamics. Building on this asymptotic analysis, we develop a simple yet effective numerical algorithm to address ill-posed nonlinear inverse problems aimed at reconstructing coefficient functions that depend solely on spatial or temporal variables. Conditions ensuring the existence and uniqueness of solutions for both forward and inverse problems are established. The proposed method's effectiveness is validated through numerical experiments, demonstrating high accuracy in reconstructing coefficient functions under varying noise conditions.
\end{abstract}


\markright{Coefficient recovery in time-periodic RDA equations}

\footnotetext[0]{This work was partially supported by the National Natural Science Foundation of China (No. 12350410359 \& No. 12171036), National Key Research and Development Program of China (No. 2022YFC3310300) and Beijing Natural Science Foundation (No. Z210001).}


\section{Introduction}

Reaction-diffusion-advection equations are fundamental in modeling dynamic phenomena in a variety of scientific disciplines, including biological systems \cite{Cos2014}, chemistry \cite{ComNegBrau2023}, and environmental sciences \cite{ZhaForGul2016, ZhaLinGul2017}. These equations become particularly insightful when dealing with small parameter values, which help to elucidate processes from  cell polarization \cite{Emken2020} to solute transport \cite{Erofeev2021}. This study focuses on reaction-diffusion-advection models where chemical reactions occur at a much faster rate than diffusion, resulting in sharp concentration gradients. Such models fall under the category of singularly perturbed PDEs, where the presence of a small parameter $\mu$ significantly alters the solution's behavior as $\mu$ approaches zero. We focus specifically on solutions characterized by moving fronts, which are critical for understanding phenomena such as the spread of environmental pollutants \cite{ManPim2020}.

To analyze these systems, we employ a matched asymptotic expansion method with multi-scale analysis, following the conventional techniques of asymptotic analysis \cite{Nefedov2013, Levashova2018}. This approach involves solving the PDE in two distinct regions: (a) outer solution: Valid away from regions where sharp gradients occur, where diffusion is relatively less significant compared to reaction; and (b) inner solution: Valid in regions with significant diffusion, where boundary layers or rapid transitions are observed. We then match the inner and outer solutions in overlapping regions to obtain a uniform approximation across the entire domain.  

In addition to studying the well-posedness and properties of the reaction-diffusion- advection equations, we address a parameter identification problem associated with these singularly perturbed systems. This inverse problem is nonlinear and ill-posed, posing challenges in practical applications.  Our findings have important implications for fields such as environmental pollution monitoring \cite{amp20221012}, water resource management \cite{Liang2010}, thermal oil recovery \cite{Got1965},  multiphase porous media flow \cite{Liu2023}, and materials science \cite{Tilioua2012}. Solving inverse problems and accurately determining unknown coefficients are crucial in these domains. Addressing the complexity of such problems, especially in identifying coefficient functions, remains a significant challenge (for further details, see \cite{Isa1990}).

As demonstrated in Section \ref{Statement}, numerous mathematical formulations exist for a given inverse problem. Some straightforward formulations, while easy to understand, can be challenging to solve in practice. In this paper, we use asymptotic analysis to derive two simple \textit{link equations}, cf. \eqref{preApprF0x} and \eqref{Eqf(t)}, between measurements and unknown parameters, significantly reducing the complexity of the original inverse problems. Our rigorous analysis proves that this model reduction is sufficiently accurate even in the presence of noise. This high-level idea has been applied recently to various inverse problems in perturbed PDEs; see, e.g. \cite{ChaZha2022, Chaikovskii2D2023, Chaikovskii3D2023, Lukyanenko2019, Lukyanenko2019v2, Lukyanenko2020, Lukyanenko2021, Lukyanenko2021v2}. 

The main contributions of this paper include the restoration of the coefficient function from the right-hand side, which may be time-dependent and periodic, and the formulation of conditions for the existence of solutions to the inverse problem. We also estimate the convergence rates for the inverse problem and derive the initial conditions necessary for numerical solutions using asymptotic analysis. Moreover, we established precise estimates for the location of the transition layer \( x_{t.p.}(t) \) by proving that it deviates from its leading-order approximation \( x_{0}(t) \) by at most \( C \mu |\ln \mu| \). This result provides accurate information about the position of the transition layer, enhancing the reliability of our asymptotic and numerical approximations. This research aims to advance the field by proving the existence of a unique, smooth solution for non-stationary, singularly perturbed differential equations with time-dependent right-hand sides. Our approach offers tractable approximations and valuable analytical insights into complex systems. We have developed a novel periodic asymptotic solution that incorporates an inner transition layer, enhancing the mathematical modeling of reaction-diffusion-advection dynamics. Furthermore, we have extended our asymptotic analysis algorithm to address inverse problems related to various coefficient functions. By deriving conditions for the existence and uniqueness of both forward and inverse solutions, we enhance the reliability and precision of our methods. In summary, this research bridges theoretical mathematical frameworks with practical applications, advancing the understanding and application of reaction-diffusion-advection equations across critical scientific and engineering fields.

This paper is structured as follows: Section \ref{Statement} introduces the mathematical formulations of the direct problem and two associated inverse problems, specifying the equations and assumptions used. Section \ref{asymptoticanalysis} provides the mathematical background and theoretical foundation necessary for the analysis. This section elaborates on the asymptotic expansion method, detailing the construction of outer and inner functions, as well as the derivation of equations governing the transition layer. Section \ref{mainresults} introduces the key theoretical results, including theorems on the existence and uniqueness of solutions, asymptotic estimates, and the development of regularization methods for addressing the inverse problems. Section \ref{examples} describes the numerical methods and computational techniques used to implement these theoretical results, offering examples that illustrate the practical application of the proposed methods in solving both direct and inverse problems under various scenarios. The paper concludes in Section \ref{conclusion}, with detailed proofs for certain claims provided in the appendices.

	\section{Problem statement and assumptions}
    \label{Statement}
    
	In this paper, we delve into key characteristics of the following time-dependent nonlinear reaction-diffusion-advection equation, subject to specific boundary and temporal conditions:
	\begin{equation}
    \label{InitProb}
		\begin{cases}
			\displaystyle\mu \left( \pdv[2]{u}{x} - \pdv{u}{t} \right) = -u\pdv{u}{x}+F(x,t)u, & x \in (0,1),\ t \in \mathbb{R}_+,\\
			u(0,t,\mu)=u^{0}(t),\ u(1,t,\mu)=u^{1}(t), & t \in \mathbb{R}_+,\\
			u(x,t,\mu)=u(x,t+T,\mu), & x \in (0,1),\ t \in \mathbb{R}_+,
		\end{cases}
	\end{equation}
where \(\mu\) represents a small parameter, emphasizing the equation's sensitivity and the perturbative nature of the study. The coefficient function \(F(x, t)\) is defined to be \(T\)-periodic and is assumed to maintain adequate smoothness across the domain \(\bar{D} = (0,1) \times \mathbb{R}_+\), mirroring the periodicity and smoothness requirements for \(u^{0}(t)\) and \(u^{1}(t)\).

In this work, we investigate the following direct problem (\textbf{DP}) and two inverse problems (\textbf{IP1} and \textbf{IP2}).

\begin{itemize}
\item \textbf{Direct Problem} (\textbf{DP}): Given functions $u^0(t)$, $u^1(t)$ and $F(x,t)$, study the well-posedness of non-linear PDE \eqref{InitProb}. 

\item \textbf{Inverse Problem 1} (\textbf{IP1}):  Given noisy gradient data $\{w^\delta_i \}_{i=0}^k$ of measurements $\{\frac{\partial u}{\partial x}(x_i,t_0)\}_{i=0}^k$ or noisy data $\{ u^\delta_i \}_{i=0}^k$ of $\{ u(x_i,t_0)\}_{i=0}^k$ for any fixed time $t_0 \in \mathbb{R}_+ $ on the grid
	\begin{equation*}
    \label{gridX}
		\Theta_x = \bigg\{ x_i=\dfrac{i}{k}: \ 0\leq i\leq k \bigg\}, 
	\end{equation*}
such that 
\begin{equation}
\label{noise1}
\max_{0\leq i\leq k} \bigg\{ |u^\delta_i - u(x_i,t_0)|, |w^\delta_i - \frac{\partial u}{\partial x}(x_i,t_0)| \bigg\} \leq \delta,
\end{equation}
recover the time-independent coefficient function $F(x,t)\equiv f(x)$ in \eqref{InitProb}.

\item \textbf{Inverse Problem 2} (\textbf{IP2}): Given noisy data $\{ u^{\delta}_{t.p.,i}, u^{1,\delta}_i, x^\delta_{0,i} \}_{i=0}^k$ of measurements $\{u^0(t_i), u^1(t_i), x_{t.p.}(t_i)\}_{i=0}^k$ ($x_{t.p.}(t_i)$ denotes the location of transition layer, with a rigorous definition of $x_{t.p.}(t)$ provided in Definition \ref{Defin1}) on the grid  
\begin{equation*}
    \label{gridT}
		\Theta_t=\bigg\{t_i=\dfrac{i \cdot T}{k}: \ 0\leq i\leq k \bigg\}
	\end{equation*}
such that 
\begin{equation}
\label{noise2}
\max_{0\leq i\leq k} \bigg\{ |u^{0,\delta}_i - u^0(t_i)|, |u^{1,\delta}_i - u^1(t_i)|, |x^\delta_{t.p.,i} - x_{t.p.}(t_i)| \bigg\} \leq \delta,
\end{equation}
recover the spatial-independent coefficient function $F(x,t)\equiv f(t)$ in \eqref{InitProb}. 
\end{itemize}

The two inverse problems under consideration present significant challenges in terms of solvability and uniqueness. Without additional assumptions, uniqueness is unattainable due to the discrepancy between the finite data space and the infinite solution space. However, uniqueness can be established under specific conditions (see, e.g., \cite{Shishatskii1988,Bukhgeim1993,Tataru2004,Isa1990}). These additional assumptions are crucial for sufficiently constraining the problem to ensure a unique solution. Furthermore, due to the unbounded nature of inverse differential operators, the inverse problems we consider, especially when dealing with noisy data, are ill-posed. This implies that small perturbations in the data can lead to significant deviations in the results obtained through conventional methods. To address the aforementioned issues of ill-posedness, a straightforward variational formulation of the inverse problems, incorporating Tikhonov regularization techniques, is typically employed. This approach transforms the problems (\textbf{IP1}) and (\textbf{IP2}) into the following two optimization problems with partial differential equation (PDE) constraints:
\begin{equation}
\label{control1}
\min_{F(x,t), u(x,t)} \sum^{k}_{i=1}   \left[ \lambda (u^\delta_i - u(x_i,t_0))^2 + (1-\lambda) (w^\delta_i - \frac{\partial u}{\partial x}(x_i,t_0))^2 \right] + \varepsilon\mathcal{R}(F),
\end{equation}
and 
\begin{equation}
\label{control2}
\min_{F, u} \sum^{k}_{i=1}  \left[ (u^{0,\delta}_i - u^0(t_i))^2 + (u^{1,\delta}_i - u^1(t_i))^2 + (x^\delta_{0,i} - x_0(t_i))^2 \right] + \varepsilon\mathcal{R}(F),
\end{equation}
where the state $u(x,t)$ solves the PDE \eqref{InitProb} with a given coefficient function $F(x,t)$, $\mathcal{R}(F)$ is the regularization term,  describing the a priori information of function $F(x,t)$, and $\varepsilon > 0$ is the regularization parameter. In optimization problem \eqref{control1}, $\lambda=0$ for gradient data $\{w^\delta_i \}_{i=0}^k$, while $\lambda=1$ when we have only data $\{ u^\delta_i\}_{i=0}^k$. 

However, this approach presents two significant challenges: selecting an appropriate regularization term and parameter and solving the resulting non-convex optimization problem with nonlinear PDE constraints. To address these issues, in this work, a new inverse method based on asymptotic analysis is proposed, aiming to develop a simpler yet accurate model that can effectively replace the complex PDE-constrained optimization problems \eqref{control1} and \eqref{control2}. 

To this end, we list some assumptions for our direct problem (\textbf{DP}) and two inverse problems (\textbf{IP1}) and (\textbf{IP2}). 

First, to ensure the integrity of solutions of (\textbf{DP})  and adherence to physical and mathematical principles, we establish the necessary conditions:
	
	\begin{assumption}\label{Cond1}
		Boundary values must comply with the inequalities ensuring physical consistency and mathematical stability:
		$$\max_{t \in \mathbb{R}_+} \left( u^0(t) \right) < \min \left\{0, -\max_{\substack{x \in [0,1] \\  t \in \mathbb{R}_+}} \left( \int\limits_{0}^{x} F(\tau,t)\, \mathrm{d}\tau \right) \right\} , $$ 
             $$ \min_{ t \in \mathbb{R}_+} \left(  u^1(t) \right) > \max \left\{0,\max_{\substack{x \in [0,1] \\  t \in \mathbb{R}_+}} \left( \int\limits_{x}^{1} F(\tau,t)\, \mathrm{d}\tau \right)\right\}, $$ $$ u^1(t) - u^0(t) >  2\mu^\theta, \quad  \text{for some fixed~} \theta\in(0,1), \quad \forall t \in \mathbb{R}_+. $$
	\end{assumption}

	\begin{assumption}\label{Cond2}
		We require a smooth \(t\)-periodic solution \(x = x_0(t)\) satisfying the following for the transition dynamics to be physically viable:
		\[
		I(x_0(t),t): = \qty(\varphi^{l}(x_0(t),t))^2 - \qty(\varphi^{r}(x_0(t),t))^2 = 0,  
		\]
		\[
		\pdv{ I(x_0(t),t)}{x_0} < 0, \quad 0 < x_0(t) < 1, 
		\]
		ensuring \(x_0(t)\) serves as a unique root of $I(x_0(t),t)=0$  for all \(t \in \mathbb{R}_+\).
	\end{assumption}
	
	\begin{assumption}\label{Cond3}
		The initial state of our moving front is defined as:
		\[
		u_{init}(x) = U_{0}(x, 0) + \mathcal{O}(\mu),
		\]
		with \(U_{0}\) introducing an initial asymptotic approximation for the moving front's shape and location.
	\end{assumption}

Under Assumptions \ref{Cond1}-\ref{Cond3}, by using the asymptotic analysis theory, we can ensure the existence of a unique smooth solution for the initial problem \eqref{InitProb}, cf. Theorems \ref{existenceTheorem} and \ref{asymptEstimates}. 

Note that for the inverse problem with the coefficient function depending only on the spatial variable, i.e.  \(F(x,t)=f(x)\), Assumptions \ref{Cond1}-\ref{Cond3} are sufficient. However, if the coefficient function is time-dependent (in this paper, we only consider the case \(F(x,t)=f(t)\)), the following additional assumption for precise transition layer localization is required: 
	
	\begin{assumption}\label{Cond4}
    There exists a number $a\in(0, 0.5)$ such that the position of the transition layer is restricted to:
		\[
		0 < x_{t.p.}(t) < 0.5-a \quad \text{or} \quad 0.5+a < x_{t.p.}(t) < 1, \quad t \in \mathbb{R}_+,
		\]
		ensuring a definite and physically meaningful placement within the spatial domain.
	\end{assumption}

	\section{Asymptotic analysis}
    \label{asymptoticanalysis}
	\subsection{Asymptotic representation of the solution}
	
	In this analysis, we introduce the standard asymptotic approximation for the solution to the problem delineated in \eqref{InitProb}:
	\begin{equation*}\label{U}
		U(x,t,\mu)=
		\begin{cases}
			U^l(x,t,\mu), & x\in \overline{\Omega}^l,\\
			U^r(x,t,\mu), & x\in \overline{\Omega}^r.
		\end{cases}
	\end{equation*}
	where
	\begin{equation}\label{AsySol}
		U^{l,r}=\bar{u}^{l,r}(x,t,\mu)+Q^{l,r}(\xi,t,\mu),
	\end{equation}
	and 
	\begin{align*}
		\overline{\Omega}^l := \{(x,t) : 0 \leqslant x \leqslant x_{t.p.}(t), t \in \mathbb{R}_+\},\\
		\overline{\Omega}^r := \{(x,t) : x_{t.p.}(t) \leqslant x \leqslant 1, t \in \mathbb{R}_+\}.
	\end{align*}
	In accordance with conventional theoretical practices, \(\bar{u} (x,t,\mu)\) denotes outer functions that depict the solution's behavior distant from the curve \(x_{t.p.}(t)\), and the inner functions \(Q (\xi,t,\mu)\), which describe the transition layer, with the stretched variable
	\begin{equation}\label{ScaledVar}
		\xi=\frac{x-x_{t.p.}(t)}{\mu},
	\end{equation}
	characterize it within a proximate neighborhood of \(x_{t.p.}(t)\).
	
	For areas excluding the boundary, we employ analogous notation:
	\begin{align*}
		\Omega^l := \{(x,t) : 0 < x < x_{t.p.}(t), t \in \mathbb{R}_+\},\\
		\Omega^r := \{(x,t) : x_{t.p.}(t) < x < 1, t \in \mathbb{R}_+\}.
	\end{align*}
	
	Our objective is to derive the solution for the problems presented in \eqref{InitProb}, characterized by moving front behavior. Such situations arise when, at each temporal moment within the interval \(0 \leqslant x \leqslant x_{t.p.}(t)\), the solution approximates \(\varphi^l(x, t)\), and within the interval \(x_{t.p.}(t) \leqslant x \leqslant 1\), it approximates \(\varphi^r(x, t)\). Consequently, there is a rapid transition of the solution from \(\varphi^l(x, t)\) to \(\varphi^r(x, t)\) in the vicinity of the curve \(x = x_{t.p.}(t)\). This is commonly termed as a solution with an internal layer. Following the idea in \cite{Antipov2014}, we formulate the following definition of the transition point, which will be used in the sharp error estimation of the quantity $|x_{t.p.}(t)-x_{0}(t)|$.   

    \begin{definition}\label{Defin1}
	The transition point \( x_{t.p.}(t) \) is defined as a point where the solution \( u(x, t, \mu) \) equals the average of the left and right outer solutions:
	\begin{equation}\label{barUatTP}
		{u}(x_{t.p.}(t), t,\mu) = \varphi(x_{t.p.}(t), t) := \frac{1}{2}(\varphi^l(x_{t.p.}(t),t)+\varphi^r(x_{t.p.}(t),t)), \ t \in \mathbb{R}_+.
	\end{equation}
    \end{definition}

    \begin{remark}
The point \( x_{t.p.}(t) \), defined by equation \eqref{barUatTP}, represents the location of the internal layer or moving front in the solution \( u(x, t, \mu) \). In general, multiple transition points may exist as per Definition \ref{Defin1}. However, for our specific equation \eqref{InitProb}, the transition point \( x_{t.p.}(t) \) suffices for our analysis, regardless of potential non-uniqueness. This is because it is sufficient that such a point lies within the transition layer region and satisfies the estimates \eqref{estim2} with respect to \( x_0(t) \).
    \end{remark}

	Hence, we infer that \(u(x)\) possesses an internal layer in the proximity of the curve \(x_{t.p.}\).
	
	The transitional curve is presented as an asymptotic series:
	\begin{equation*}\label{asyTP}
		X_{n}(t,\mu) = x_0(t) + \mu x_1(t) + \ldots + \mu^{n}x_{n}(t).
	\end{equation*}
	
	The functions \(\bar{u}^{l,r}(x,t,\mu)\) and \(Q^{l,r}(\xi,t,\mu)\) in \eqref{AsySol} are similarly represented in an asymptotic form:
	\begin{align}
		\label{asySerReg}
		\bar{u}^{l,r}(x,t,\mu)=\bar{u}^{l,r}_0(x,t)+\mu\bar{u}^{l,r}_1(x,t)+\ldots+\mu^n\bar{u}^{l,r}_n(x,t)+\ldots,\\
		\label{asySerInner}
		Q^{l,r}(\xi,t,\mu)=Q^{l,r}_0(\xi,t)+\mu Q^{l,r}_1(\xi,t)+\ldots+\mu^nQ^{l,r}_n(\xi,t)+\ldots .
	\end{align}
	
	The boundary condition for the functions  \(U^l(x,t,\mu)\) and \(U^r(x,t,\mu)\) could be define as
	\[
	U^l(x_{t.p.}(t),t,\mu) = U^r(x_{t.p.}(t),t,\mu) = \varphi(x_{t.p.}(t),t),
	\]
	and their derivatives normal to the curve \(x = x_{t.p.}(t)\), which is as defined in \eqref{barUatTP}, must be continuously matched along the curve \(x_{t.p.}(t)\) at every temporal instant \(t\):
	\begin{equation}\label{derSew}
		\pdv{U^l}{x} \qty(x_{t.p.}(t),t) = \pdv{U^r}{x} \qty(x_{t.p.}(t),t).
	\end{equation}

	\subsection{Outer functions}
	Initially, let us consider the equation derived from the initial problem delineated in \eqref{InitProb}:
	\begin{equation} \label{InitEq}
		\mu\qty(\pdv[2]{u}{x} - \pdv{u}{t} ) + u\pdv{u}{x} - F(x,t)u = 0.
	\end{equation}
	
	Substituting \eqref{asySerReg} into \eqref{InitEq} and equating terms with identical powers of \(\mu\), we obtain equations defining the outer functions \(\bar{u}_i\) in zero-order approximation:
	\begin{equation*}\label{regEqs}
		\begin{aligned}
			-F(x,t)\bar{u}_0^{l,r} + \bar{u}_0^{l,r} \pdv{\bar{u}_0^{l,r}}{x} = 0.
		\end{aligned}
	\end{equation*}
	
	Incorporating the boundary conditions from the initial problem \eqref{InitProb}, we stipulate:
	\begin{equation}\label{regProbL0}
		\begin{cases}
			\displaystyle \pdv{\bar{u}_0^l}{x} = F(x,t) , & x\in [0, 1],\; t \in \mathbb{R}_+,\\
			\bar{u}_0^l(0,t) = u^{0}(t), & t \in \mathbb{R}_+,
		\end{cases}
	\end{equation}
	and
	\begin{equation}\label{regProbR0}
		\begin{cases}
			\displaystyle \pdv{\bar{u}_0^r}{x} = F(x,t) , &  x\in [0, 1],\; t \in \mathbb{R}_+,\\
			\bar{u}_0^r(1,t) = u^{1}(t), & t \in \mathbb{R}_+.
		\end{cases}
	\end{equation}
	
	From the established equations \eqref{regProbL0} and \eqref{regProbR0}, the solutions can be deduced as follows:
	\begin{equation}\label{regSolL}
		\bar{u}_0^l(x,t) = \varphi^l(x,t) = u^{0}(t) + \int\limits_{0}^{x} F(\tau,t)\dd{\tau},\quad x\in [0, 1],\; t \in \mathbb{R}_+,
	\end{equation}
	\begin{equation}\label{regSolR}
		\bar{u}_0^r(x,t) = \varphi^r(x,t) = u^{1}(t) - \int\limits_{x}^{1} F(\tau,t)\dd{\tau},\quad x\in [0, 1],\; t \in \mathbb{R}_+.
	\end{equation}

	\subsection{The inner functions}
	Next, we consider equation \eqref{InitEq} in both regions $\Omega^{(l, r)}$ with corresponding boundary and initial conditions: 
	\begin{align*}
		&u(0,t) = u^{0}(t),\quad u(x_{t.p.}(t),t) = \varphi(x_{t.p.}(t),t), \quad &\text{for area $\Omega^l$}, \\
		&u(x_{t.p.}(t),t) = \varphi(x_{t.p.}(t),t),\quad u(1, t) = u^{1}(t), \quad &\text{for area $\Omega^r$},\\
		&u(x,t)=u(x,t+T), \qquad &x\in [0, 1],\; t \in \mathbb{R}_+,
	\end{align*}
	where function $\varphi(x_{t.p.}(t),t)$ is defined in \eqref{barUatTP}. 
	
	For simplicity, we introduce the following operator $L_{x}[u, \mu] := \displaystyle\mu\qty(\pdv[2]{u}{x} - \pdv{u}{t})$. It is possible to rewrite it in terms of $\xi$, which is defined as \eqref{ScaledVar}:
	\begin{equation}\label{Lmu}
		L_{\xi}[u, \mu] := \frac{1}{\mu} \, \pdv[2]{u}{\xi} +x'_{t.p.}(t)\pdv{u}{\xi} -\mu\pdv{u}{t}.
	\end{equation}
	
	Substitution of \eqref{Lmu} into equation \eqref{InitEq} gives us the following expression:
	\[
		Lf_{\xi}[u, \mu] := \frac{1}{\mu}\pdv[2]{u}{\xi} + x'_{t.p.}(t)\pdv{u}{\xi}- {\mu}\pdv{u}{t}+\frac{1}{\mu}u\pdv{u}{\xi}=f(\mu\xi + {x}_{t.p.}(t),t)u.
	\]
	
	By finding the difference $Lf_{\xi}[\bar{u}+Q, \mu] - Lf_{\xi}[\bar{u}, \mu] $, let us define the equation for $Q$:
	\begin{equation}\label{QEq}
		\frac{1}{\mu}\pdv[2]{Q}{\xi} + x_0'(t)\pdv{Q}{\xi}- {\mu}\pdv{Q}{t}+\frac{1}{\mu}{\bar{u}}\pdv{Q}{\xi} + \frac{1}{\mu}Q\pdv{(\bar{u}+Q)}{\xi} = f(\mu\xi + {x}_{t.p.}(t),t)Q.
	\end{equation}
	
	Let us split functions into Taylor's series of the powers of $\mu$ and substitute them with \eqref{asySerReg} and \eqref{asySerInner}  into \eqref{QEq}, and equate the coefficients for the small parameter $\mu^{-1}$. And since $ \varphi^{l,r}$ are not depending on $\xi $, we obtain an equation for the inner functions $Q_0$ in the zero-order approximation:
	\begin{equation}\label{QEq0}
		\pdv[2]{(\varphi^{l,r}+Q_0^{l,r})}{\xi} = -(\varphi^{l,r}+Q_0^{l,r})\pdv{(\varphi^{l,r}+Q_0^{l,r})}{\xi}.
	\end{equation}
	
	At this point, it is convenient to introduce the notation:
	\begin{equation}\label{tildaU}
		\tilde{u}(\xi, t) =
		\begin{cases}
		\begin{aligned}
			&\varphi^l(x_0(t),t)+Q_0^l(\xi,t), \quad &\xi<0,\\
			&\varphi(x_0(t),t), \qquad &\xi=0,\\
			&\varphi^r(x_0(t),t)+Q_0^r(\xi,t), \quad &\xi>0,
		\end{aligned}
		\end{cases}
	\end{equation}
	where $ \varphi^l<\tilde{u}<\varphi^r$. Here we should note that in the case of zero approximation, the equation \eqref{barUatTP} changes to
	\[
		\varphi(x_0(t),t)=\frac{\varphi^l(x_0(t),t)+\varphi^r(x_0(t),t)}{2}.
	\]
	
	By using replacements \eqref{tildaU} and $\displaystyle\frac{\partial \tilde{u}}{\partial\xi}= g(\tilde{u})$ we can reduce equation \eqref{QEq0} to the next form:
	\begin{equation}\label{gOfTildaU}
		\pdv{g(\tilde{u})}{\tilde{u}} = -\tilde{u}.
	\end{equation}
	
	Let us integrate the left and right sides of \eqref{gOfTildaU} when $\xi<0$:
	\[
		\int_{\varphi^l}^{\tilde{u}} \pdv{g(\tau)}{\tau} \dd{\tau} \equiv g(\tilde{u})-g(\varphi^l)  = -\int_{\varphi^l}^{\tilde{u}} \tau \dd{\tau}.
	\]
	
	Since $ \varphi^{(l)}$ does not depend on $\xi $, then $g(\varphi^l)=0$. Therefore,  it is possible to find  
	\begin{equation}\label{QL0Int}
		\pdv{Q_0^l}{\xi}=\pdv{\tilde{u}}{\xi} = g(\tilde{u}) = -\int_{\varphi^l}^{\tilde{u}} \tau \dd{\tau}.
	\end{equation}
	
	Similarly, for $\xi>0$ we have:
	\begin{equation}\label{QR0Int}
		\pdv{Q_0^r}{\xi}=\pdv{\tilde{u}}{\xi} = g(\tilde{u}) = \int_{\tilde{u}}^{\varphi^r} \tau \dd{\tau}.
	\end{equation}
	
	Let us write the matching conditions \eqref{derSew} in the zero approximation:
    \[
		\pdv{Q_0^r}{\xi} - \pdv{Q_0^l}{\xi} = 0.
	\]
	
	Taking into account \eqref{QL0Int} and  \eqref{QR0Int}, we obtain:
	\begin{equation}\label{halfSqrtDifference}
		\int_{\varphi^l}^{\tilde{u}} \tau \dd{\tau} + \int_{\tilde{u}}^{\varphi^r} \tau \dd{\tau} = \int_{\varphi^l(x_0(t),t)}^{\varphi^r(x_0(t),t)} \tau d\tau = \frac{1}{2}\qty({(\varphi^r(x_0(t),t))^2}-{(\varphi^l(x_0(t),t))^2}) = 0.
	\end{equation}

	Let us substitute equations \eqref{regSolL} and \eqref{regSolR} into \eqref{halfSqrtDifference}:
	\begin{equation}\label{findx0}
		\frac{1}{2}\qty(\qty(u^{1}(t)+\int_{1}^{x_0(t)} F(\tau,t)d\tau)^2 - \qty(u^{0}(t) + \int_{0}^{x_0(t)} F(\tau,t)d\tau)^2)=0.
	\end{equation}
	
	From Assumption \ref{Cond2} follows the existence of solution $x_0(t)$ to the equation \eqref{findx0}, which describes the position of the middle of the transition layer, necessary for constructing an asymptotic solution. In the general case, this equation is solved numerically. If $F(x,t)=f(t) $ we can write the formula for $x_0(t)$ explicitly:
	\begin{equation}\label{x0}
		x_0(t) = -\frac{u^{0}(t)+u^{1}(t)-f(t)}{2f(t)}.
	\end{equation}
	
	To calculate $Q_0(\xi,t)$, we integrate equations \eqref{QL0Int} and \eqref{QR0Int} substituting the value of the auxiliary function $\tilde{u}$; thus we obtain:
	\[
		\pdv{Q^{l,r}(\xi,t)}{\xi} = -\frac{1}{2}(Q^{l,r}(\xi,t)^2 + 2\varphi^{l,r}(x_0(t),t) Q^{l,r}(\xi,t) ).
	\]
	
	If we pair them with the following conditions
	\begin{align*}
		Q_0^l(-\infty,t)=0,\quad Q_0^l(0,t)=\varphi(x_0(t),t)-\varphi^l(x_0(t),t),\\
		Q_0^r(+\infty,t)=0,\quad Q_0^r(0,t)=\varphi(x_0(t),t)-\varphi^r(x_0(t),t),
	\end{align*}
	
	we can find $Q_0^l(\xi,t)$ and $Q_0^r(\xi,t)$ explicitly: \\
	\begin{equation*}\label{coefQ0L}
		Q_0^l(\xi,t) = \dfrac{-2\varphi^l(x_0(t),t)\displaystyle \exp(-\Phi_l)}{\displaystyle \exp(-\Phi_l) - \exp(\varphi^l(x_0(t),t)\xi)},
	\end{equation*}
	\begin{align*}
		Q_0^r(\xi,t)=\dfrac{-2\varphi^r(x_0(t),t)\displaystyle \exp(-\Phi_r)}{ \exp(-\Phi_r) - \exp(\varphi^r(x_0(t),t)\xi)},
	\end{align*}
	where
	\[
		\Phi_l = \ln\frac{\varphi^r+3\varphi^l}{\varphi^r-\varphi^l}, \qquad
		\Phi_r = \ln\frac{\varphi^l+3\varphi^r}{\varphi^l-\varphi^r}.
	\]
	
	So, we can write the expression for the asymptotic solution in the zero approximation:
	\begin{multline*}\label{solU0}
		U_0 (x,t)= \\
        \begin{cases}
			\varphi^l(x,t)+\displaystyle\frac{-2\varphi^l(x_0(t))}{1-\displaystyle\frac{\varphi^r(x_0(t))+3\varphi^l(x_0(t))}{\varphi^r(x_0(t))-\varphi^l(x_0(t))}\displaystyle \exp(\varphi^l(x_0(t))\displaystyle\frac{x-x_0(t)}{\mu})}, \  x \in [0,x_0(t)], \\[1cm]
			\varphi^r(x,t)+\displaystyle\frac{-2\varphi^r(x_0(t))}{1-\displaystyle\frac{\varphi^l(x_0(t))+3\varphi^r(x_0(t))}{\varphi^l(x_0(t))-\varphi^r(x_0(t)}\displaystyle \exp(\varphi^r(x_0(t))\displaystyle\frac{x-x_0(t)}{\mu})}, \ x \in [x_0(t),1] .
		\end{cases}
	\end{multline*}

	\section{Main results}
    \label{mainresults}

\subsection{Well-posedness of (\textbf{DP})}
    
	\begin{theorem}[(Existence and asymptotic solution)]
\label{existenceTheorem}
Suppose that $F(x,t)\in C^{n+3,n+2}$ $([0,1], \mathbb{R})$, where $n$ is an order of asymptotic approximation, $u_{init} (x) \in C^{n+3}[0,1]$, $u^{0}(t),u^{1}(t) \in C^{n+2}( \mathbb{R})$, and $\mu \ll 1$. If Assumptions \ref{Cond1} - \ref{Cond3} are satisfied, then the boundary value problem \eqref{InitProb} has a unique smooth solution with an internal transition layer. 
\end{theorem}

\begin{theorem}[(Asymptotic estimates)]
\label{asymptEstimates}
Under the same conditions as Theorem \ref{existenceTheorem}, the solution of the boundary value problem \eqref{InitProb} can be approximated by the asymptotic series:
\begin{equation*}\label{mainSol}
	U_{n}(x,t,\mu)=
	\begin{cases}                     
		U_{n}^l(x,t,\mu)=\sum\limits_{i=0}^n\mu^i(\bar{u}_i^l(x)+Q_i^l(\xi_n,t)),\quad(x,t)\in[0,X_{n}(t,\mu)]\times\mathbb{R}_+,\\[8pt]
		U_{n}^r(x,t,\mu)=\sum\limits_{i=0}^n\mu^i(\bar{u}_i^r(x)+Q_i^r(\xi_n,t)),\quad(x,t)\in[X_{n}(t,\mu),1]\times\mathbb{R}_+,
	\end{cases} 
\end{equation*}
where $\xi_n=\frac{x-X_{n}(t,\mu)}{\mu}$.

Furthermore, it has the next asymptotic estimates:
\begin{align}
	&\forall(x,t)\in [0,1]\times\mathbb{R}_+:|u(x,t)-U_{n}(x,t,\mu)|\leqslant \mathcal{O}(\mu^{n+1}), \label{estim1} \\
	&\forall (x,t) \in [0,1] \backslash \{x_{t.p.}(t)\}\times\mathbb{R}_+: \abs{\pdv{u(x,t)}{x} - \pdv{U_n(x,t,\mu)}{x}} \leqslant \mathcal{O}(\mu^{n}). \label{estim3}
\end{align}

In the zero-order case, $U_0$ could be defined as
\begin{equation*}
	U_0(x,t)=
	\begin{cases}                     
		\varphi^l(x,t)+Q_0^l(\xi_0,t),\quad(x,t)\in[0,x_0(t)]\times\mathbb{R}_+,\\[5pt]
		\varphi^r(x,t)+Q_0^r(\xi_0,t),\quad(x,t)\in[x_0(t),1]\times\mathbb{R}_+.
	\end{cases} 
\end{equation*}

For the zero-order asymptotic solution there exist constants $C$ independent of $\mu$, $x$, and $t$ within our domain (excluding the narrow transition area, where the width $\Delta x=\hat{x}^r(t, \mu)-\hat{x}^l(t, \mu) \sim \mu | \ln \mu|$, with $\hat{x}^r(t, \mu)$ and $\hat{x}^r(t, \mu)$  representing the boundaries of the transition layer and defined in \eqref{hatxl}, \eqref{hatxr}), such that:
\begin{align} 
	&\abs{u(x,t)-\varphi^l(x,t)} \leqslant C\mu, \quad (x,t) \in [0,\hat{x}^l(t, \mu)] \times \mathbb{R}_+, \label{leftestimatevarphi}\\
	&\abs{u(x,t)-\varphi^r(x,t)} \leqslant C\mu, \quad (x,t) \in [\hat{x}^r(t, \mu),1] \times \mathbb{R}_+,\label{rightestimatevarphi}\\
	&\abs{\pdv{u(x,t)}{x} - \dv{\varphi^l(x,t)}{x}} \leqslant C\mu, \quad (x,t) \in [0,\hat{x}^l(t, \mu)] \times \mathbb{R}_+,\\
	&\abs{\pdv{u(x,t)}{x} - \dv{\varphi^r(x,t)}{x}} \leqslant C\mu, \quad (x,t) \in [\hat{x}^r(t, \mu),1] \times\mathbb{R}_+. \label{estim6}
\end{align}

Moreover, there exists an additional estimate for the location of the transition layer:
\begin{align} 
 &\forall t \in \mathbb{R}_+:|x_{t.p.}(t)-x_{0}(t)|\leqslant C_x\mu | \ln \mu |, \label{estim2} 
\end{align}
where $C_x$ is a constant indepedent of $x$ and $t$. 
\end{theorem}

    \begin{remark}
Actually, the constant $C_x$ encapsulates the higher-order terms and could be calculated as:
$$
C_x=\sup_{t \in \mathbb{R}_+}\left|\sum_{i=1}^{\infty}\mu^i x_i(t)\right|.
$$
The general equation for the values $x_i(t) $ can be obtained from the matching condition \eqref{derSew} using an approach similar to that presented in \cite{ChaZha2022}. However, due to space limitations, the detailed derivation is omitted here.
    \end{remark}
    
      \begin{remark}
In Assumption \ref{Cond1}, the condition
$ u^1(t) - u^0(t) > 2\mu^\theta \ \text{for} \ \theta \in (0,1) $
is only necessary to establish the estimate \eqref{estim2}. According to our proof in the Appendix, for the estimates \eqref{leftestimatevarphi}$-$\eqref{estim6}, a weaker condition $ u^1(t) - u^0(t) > 2\mu^2 $ is sufficient.
    \end{remark}

	The proofs for Theorems \ref{existenceTheorem} and \ref{asymptEstimates} are quite technical and some parts are similar to proofs given in \cite{Nefedov2013, Levashova2018} and \cite{ChaZha2022}, and hence is postponed in the appendix.

\subsection{An efficient regularization method for (\textbf{IP1})}

Utilizing the asymptotic approximation of the direct problem, we designed an efficient algorithm for determining the function $f(x)$. The fundamental idea of this new algorithm is to substitute the reaction-diffusion-advection equation \eqref{InitProb} with a simplified expression~-- while maintaining the same order of accuracy~-- from the systems \eqref{regProbL0} and \eqref{regProbR0}. We refer to this type of relation as the intrinsic connection between the unknown parameter and the measurement. 

To that end, let us define the pre-approximative function:
	\begin{equation}
    \label{preApprF0x}
		f_0(x)=\frac{\partial u(x,t_0)}{\partial x}.
	\end{equation}
	
	Assume we have a deterministic model \eqref{noise1} between noisy data $\omega_i^\delta$ and accurate data $\displaystyle\pdv{u(x_i,t_0)}{x}$ for a fixed time $t_0$ and at grid nodes $\Theta_x$. The maximum grid size is $h=\max_{i\in0,\ldots,k-1}\{(x_{i+1}-x_i)\}$.
	
	Now instead of optimization problem \eqref{control1} with high order PDE constraint, we restore the original function $f^\delta(x)$ by solving the following simper minimization problem: 
	\begin{equation}\label{minProbFdelta}
		f^\delta(x)=\arg\mathop{\min}_{f\in C^1(0,1)}\frac{1}{k+1}\sum_{i=0}^{k}{(f(x_i)-\omega_i^\delta)^2}.
	\end{equation}
	
	Additionally, in cases where only noisy measurements $u_i^\delta$ are obtained, it is possible to define $\omega_i^\delta$ in \eqref{minProbFdelta} as a smoothed function $\displaystyle\pdv{u^\epsilon}{x}$. We construct $u^\epsilon(x_i,t_0)$ based on the following minimization problem: 
	\begin{equation}\label{uEps}
		u^\epsilon(x,t_0) = \arg\mathop{\min}_{s\in C^1(0,1)}\frac{1}{k+1}\sum_{i=0}^k(s(x_i,t_0)-u^\delta(t_0))^2+\epsilon(t_0)\norm{\pdv[2]{s(x,t_0)}{x}}^2_{L^2(\Omega)}.
	\end{equation}
  
	Here $\epsilon$ is the regularization parameter and it is chosen in such a way that   $u^\epsilon(x,t_0)$ satisfies the condition:
	\begin{equation*}
		\frac{1}{k+1}\sum_{i=0}^k(u^\epsilon(x_i,t_0)-u_i^\delta(t_0))^2=\delta^2.
	\end{equation*}

	\begin{lemma} 
    \label{lemma1} 
    Let $f$ be the exact function satisfying the equation \eqref{InitProb}. Then, subject to Assumption \ref{Cond2}, there exists a constant $C_1$ that does not depend on the variables $\mu$ and $x$, and it satisfies the following condition:
	\begin{equation}
    \label{estimlemma1}
		\norm{f-f_0 }_{L_p(\Omega)} \leq C_1\mu \abs{ \ln(\mu)}, \quad \forall{p\in (0,+\infty)}.
	\end{equation}
	\end{lemma}
	\begin{proof}
		Let us denote $\Omega'_l = (0, \hat{x}^l(t, \mu))$, $\Omega'_r = (\hat{x}^r(t, \mu), 1)$ and $\Omega'_m = (\hat{x}^l(t, \mu), \hat{x}^r(t, \mu))$. According to the equations \eqref{regProbL0} and \eqref{regProbR0}, the function $f$ could be defined as
		\begin{equation*}
			f=
			\begin{cases}
				\displaystyle\dv{\varphi^l(x)}{x}, \quad x\in\Omega'_l,\\[10pt]
				\displaystyle\dv{\varphi^r(x)}{x}, \quad x\in\Omega'_r.
			\end{cases}
		\end{equation*}
	
		From the asymptotic analysis, it follows that:
		\begin{equation}\label{NormFstarF0}
			\norm{f-f_0}_{L^p(\Omega'_l)}=\norm{\dv{\varphi^l(x)}{x} - \pdv{u(x,t)}{x}}_{L^p(\Omega'_l)} \leqslant c_1\mu.
		\end{equation}
		
		Following the same idea, we obtain the inequality:
		\begin{equation}\label{NormFstarF0R}
			\norm{f-f_0}_{L^p(\Omega'_r)} \leqslant c_3\mu.
		\end{equation}
		
		Since $\Delta x \sim \mu| \ln \mu|$, as it was shown in  \cite{ChaZha2022}, we obtain:
		\begin{equation}\label{fqm}
				\norm{f-f_0}_{L^p(\Omega'_m)} \leqslant \norm{f}_{L^p(\Omega'_m)} + \norm{f_0}_{L^p(\Omega'_m)} \leqslant c_2\mu \abs{\ln\mu},
		\end{equation}
		where $c_2=\norm{f}_{C(\Omega)} + \norm{f_0}_{C(\Omega)}$.
		
		From \eqref{NormFstarF0}, \eqref{NormFstarF0R} and \eqref{fqm}, it follows that:
		\begin{equation*}
        \begin{aligned}
				\left\| f-f_0\right\|_{L^p(0,1)}^p &=  \left\| f-f_0\right\|_{L^p(\Omega'_l)}^p+\left\| f-f_0\right\|_{L^p(\Omega'_m)}^p+
				\left\| f-f_0\right\|_{L^p(\Omega'_r)}
                \\
                &\leqslant \mu^p\qty(c_1^p + c_2^p\abs{\ln\mu}^p + c_3^p).
        \end{aligned}
		\end{equation*}
		
		To define \eqref{estimlemma1}, it requires an estimate with $C_1=(c_1^p+c_2^p+c_3^p)^{1/p}$.
    \end{proof}
	
	According to the results from \cite{ChaZha2022}, the following Lemma holds.
	\\
    \begin{lemma} \label{lemma2} 
        Suppose that for $t\in\mathbb{R}_+,u(\cdot,t)\in C^{2,1}(L^2(\Omega),\mathbb{R}_+) $. Let $u^\epsilon(x,t)$ be a minimizer of problem \eqref{uEps} with replacement $t = t_0$. Then for $t \in \mathbb{R}_+$ we have a constant $C_{2}$, which remains unaffected by the parameters $x$, $t$ and $\mu$, such that:
	\begin{equation}\label{uEps2}
    \begin{array}{ll}
		\norm{u^\epsilon(\cdot,t)-u(\cdot,t)}_{H^1(0,1)} &\leqslant 10\sqrt{2}\qty(h\norm{\displaystyle\pdv[2]{u(x,t)}{x}}_{L^2(\Omega)} + \sqrt{\delta}\norm{\displaystyle\pdv[2]{u(x,t)}{x}}^{1/2}_{L^2(\Omega)})
        \\
        &\leqslant C_2(h+\sqrt{\delta}).
        \end{array}
	\end{equation}
    \end{lemma}

	Given that the norms on the right-hand side of \eqref{uEps2} are particularly large in the internal layer, we proceed with an asymptotic analysis to eliminate this internal layer from our consideration. We will then separately employ the optimization problem \eqref{uEps} and minimize both regions,  $\Omega'_l$ and $\Omega'_r$, respectively. This approach allows us to minimize the error in each region independently, which is useful since the behavior of the solution differs significantly between the two regions.

	The Lemmas \ref{lemma1} and \ref{lemma2} lead us to the following Theorem:\\
	\begin{theorem} \label{theorem3} The function $f^\delta$, defined in \eqref{minProbFdelta}, is a stable approximation to the exact coefficient function $f$. Moreover, the following estimate for the rate of convergence of the inverse problem is valid:
	\begin{equation*}
		\norm{f-f^\delta}_{L^2(\Omega)} = \mathcal{O}(\mu\abs{\ln\mu} + h + \sqrt{\delta}).
	\end{equation*} 
	
	Moreover, if $\mu=\mathcal{O}(\delta^{\varepsilon+1/2})$ ($\varepsilon$ is any positive constant) and $h=\mathcal{O}(\sqrt{\delta})$, then the following equation holds:
	\begin{equation*}
		\norm{f-f^\delta}_{L^2(\Omega)} = \mathcal{O}(\sqrt{\delta}).
	\end{equation*}
    \end{theorem}

\begin{proof}
   Let's start by using the triangle inequality:

\begin{equation*}
    \|f - f^\delta\|_{L^2(\Omega)} \leq \|f - f_0\|_{L^2(\Omega)} + \|f_0 - f^\delta\|_{L^2(\Omega)}.
\end{equation*}

For the first term, we can directly apply Lemma \ref{lemma1}:

\begin{equation*}
    \|f - f_0\|_{L^2(\Omega)} \leq C_1 \mu |\ln(\mu)|.
\end{equation*}

For the second term, we proceed as follows:
\begin{align*}
\begin{split}
    \|f_0 - f^\delta\|_{L^2(\Omega)} &= \left\|\frac{\partial u(\cdot,t_0)}{\partial x} - f^\delta\right\|_{L^2(\Omega)} \\
    &\leq \left\|\frac{\partial u(\cdot,t_0)}{\partial x} - \frac{\partial u^\epsilon(\cdot,t_0)}{\partial x}\right\|_{L^2(\Omega)} + \left\|\frac{\partial u^\epsilon(\cdot,t_0)}{\partial x} - f^\delta\right\|_{L^2(\Omega)}.
\end{split}
\end{align*}

For the first part of this inequality, we can apply Lemma \ref{lemma2}, the $H^1$ norm includes the $L^2$ norm of the first derivative, so:
\begin{equation*}
    \left\|\frac{\partial u(\cdot,t_0)}{\partial x} - \frac{\partial u^\epsilon(\cdot,t_0)}{\partial x}\right\|_{L^2(\Omega)} \leq C_2(h + \sqrt{\delta}).
\end{equation*}

For the term $\|\frac{\partial u^\varepsilon}{\partial x} - f^\delta\|_{L^2(\Omega)}$, we use the properties of the minimization problem defining $f^\delta$. Since $f^\delta$ minimizes the sum of squares \eqref{minProbFdelta}, we have:
\begin{equation*}
    \left\|\frac{\partial u^\varepsilon}{\partial x} - f^\delta\right\|_{L^2(\Omega)} \leq \left\|\frac{\partial u^\varepsilon}{\partial x} - \frac{\partial u}{\partial x}\right\|_{L^2(\Omega)} \leq c_4(h + \sqrt{\delta}),
\end{equation*}
where $c_4$ is a constant.

Combining these results:
\begin{equation*}
    \|f - f^\delta\|_{L^2(\Omega)} \leq C_1 \mu |\ln(\mu)| + C_2(h + \sqrt{\delta}) + c_4(h + \sqrt{\delta}).
\end{equation*}

Therefore, we can conclude:
\begin{equation*}
    \|f - f^\delta\|_{L^2(\Omega)} = \mathcal{O}(\mu|\ln(\mu)| + h + \sqrt{\delta}).
\end{equation*}

Now, if we assume $\mu = \mathcal{O}(\delta^{\varepsilon+1/2})$ for some $\varepsilon > 0$ and $h = \mathcal{O}(\sqrt{\delta})$, we can further simplify:

\begin{align*}
\begin{split}
    \|f - f^\delta\|_{L^2(\Omega)} &= \mathcal{O}(\delta^{\varepsilon+1/2}|\ln(\delta^{\varepsilon+1/2})| + \sqrt{\delta} + \sqrt{\delta})  \\
    &= \mathcal{O}(\delta^{\varepsilon+1/2}|\ln(\delta)| + \sqrt{\delta})  \\
    &= \mathcal{O}(\sqrt{\delta}).
\end{split}
\end{align*}

The last step follows because for sufficiently small $\delta$, the term $\delta^{\varepsilon+1/2}|\ln(\delta)|$ is asymptotically smaller than $\sqrt{\delta}$.
\end{proof}

\subsection{An efficient regularization method for (\textbf{IP2})}    
    
To solve problem \textbf{IP2}, we will use the equation \eqref{x0}, which for $u^{0}(t), u^{1}(t), f(t), x_0(t)\in \mathbb{R}$ is connected by the following equation:
	\begin{equation}\label{Eqf(t)}
		f(t) = \frac{u^{0}(t)+u^{1}(t)}{1-2x_0(t)}.
	\end{equation}
	
	Assume we have a deterministic model \eqref{noise2}	at the grid nodes $\Theta_t$ with the maximum grid size $h_1=\mathop{\max}_{i\in0,\ldots,k-1}\{(t_{i+1}-t_i)\}$. Then based on \eqref{Eqf(t)} we consider the following simplified reconstruction formula:
	\begin{equation}\label{restoreeqft}
		f_i^{\delta} = \frac{u_{i}^{0,\delta}+u_i^{1,\delta}}{1-2x_{t.p.,i}^\delta}.
	\end{equation}
	%
    
Moreover, the noise level of the $f_i$ can be estimated according to the following lemma:

\begin{lemma}
\label{fiNoiseEst}
There exists a constant $C_a$, independent of the parameters $\mu$ and $t$, such that for $\delta \leq a/2$ and $\mu \leq a/(2C_x)$,
    \begin{equation*}
     \abs{f_{i}^{\delta} - f_{i}}  \leq C_a (\delta+\mu | \ln \mu|).
    \end{equation*}
\end{lemma}

\begin{proof} 
Note that $ \abs{ x_{t.p.} - x_0} \leq C_x \mu$ implies that $x_0  \leq x_{t.p.} + C_x \mu$ and $x_0  \geq x_{t.p.} - C_x \mu$. From Assumption \ref{Cond4}, we can conclude that 
\begin{equation*}
x_0(t) \leq x_{t . p .}(t)+C_x \mu| \ln \mu|<(0.5-a)+C_x \mu| \ln \mu|
\end{equation*}
or 
\begin{equation*}
x_0(t) \geq x_{t . p .}(t)-C_x \mu| \ln \mu| > (0.5+a)-C_x \mu| \ln \mu|,
\end{equation*}
which implies together with the inequality $\mu| \ln \mu| \leq \frac{a}{2 C_x}$ that 
\begin{equation*}
x_0(t)<0.5-a+C_x\left(\frac{a}{2 C_x}\right)=0.5-a+\frac{a}{2}=0.5-\frac{a}{2}
\end{equation*}
or 
\begin{equation*}
x_0(t)>0.5+a-C_x\left(\frac{a}{2 C_x}\right)=0.5+a-\frac{a}{2}=0.5+\frac{a}{2}.
\end{equation*}

As a result, we derive 
\begin{equation}
\label{x0Estimate}
1-2 x_0(t)>a \text{~or~} 1-2 x_0(t) < -a.
\end{equation}

On the other hand, from Assumption \ref{Cond4} we have: 
\begin{equation}
\label{xtpEstimate}
1-2 x_{t.p.}(t)>2 a \text{~or~} 1-2 x_{t.p.}(t)<-2 a.
\end{equation}

Therefore, when $\delta \leqslant a/2$, we have:
\begin{equation}
\label{Ineqxtp1}
1-2 x_{t.p., i}^\delta>1-2\left(x_{t.p., i}+\delta\right)>1-2 x_{t.p., i}-2 \delta>2 a-2 \delta \geqslant a>0 
\end{equation}
or 
\begin{equation}
\label{Ineqxtp2}
1-2 x_{t.p., i}^\delta<1-2\left(x_{t.p., i} - \delta\right)=1+2 \delta-2 x_{t.p., i}<-2 a+2 \delta \leqslant-a. 
\end{equation}

By combining \eqref{x0Estimate}-\eqref{Ineqxtp2}, we derive that 
\begin{equation*}
\left|\left(1-2 x^\delta_{t.p., i}\right)\left(1-2 x_{t.p.}\right)\right| > 2a^2 \text{~and~} \left|\left(1-2 x^\delta_{t.p., i}\right)\left(1-2 x_{0}\right)\right|> a^2,
\end{equation*}
which implies the following two groups of inequalities 
    \begin{equation*}
    \begin{aligned}
     &\abs{ \frac{u_{i}^{0,\delta} + u_{i}^{1,\delta}}{1 - 2x_{t.p.,i}^{\delta}} - \frac{u_{i}^{0} + u_{i}^{1}}{1 - 2x_{t.p.,i}}}  \\
     &\quad\quad =\abs{\frac{(u_{i}^{0,\delta} + u_{i}^{1,\delta})(1 - 2x_{t.p.,i}) - (u_{i}^{0} + u_{i}^{1})(1 - 2x_{t.p.,i}^{\delta})}{(1 - 2x_{t.p.,i}^{\delta})(1 - 2x_{t.p.,i})}}  \\
     &\quad\quad \leq \frac{\abs{2(x_{t.p.,i}^{\delta} - x_{t.p.,i})(u_{i}^{0} + u_{i}^{1}) + (1 -2x_{t.p.,i})(u_{i}^{0,\delta} + u_{i}^{1,\delta} - u_{i}^{0} - u_{i}^{1})}}{2a^2}  \\
     &\quad\quad \leq \frac{2(u_{i}^{0} + u_{i}^{1})\cdot \delta + 2\delta}{2a^2} = \frac{u_{i}^{0} + u_{i}^{1} + 1}{a^2}\cdot \delta,
     \end{aligned}
    \end{equation*}

and 
    \begin{equation*}
    \begin{aligned}
     \abs{ \frac{u_{i}^{0} + u_{i}^{1}}{1 - 2x_{t.p.,i}} - \frac{u_{i}^{0} + u_{i}^{1}}{1 - 2x_{0,i}}} &=  \abs{ u_{i}^{0} + u_{i}^{1}} \abs{\frac{2(x_{t.p.,i}-x_{0,i})}{(1 - 2x_{t.p.,i})(1 - 2x_{0,i})}} \\
     &\leq  \abs{ u_{i}^{0} + u_{i}^{1}} \frac{C_x}{a^2} \mu| \ln \mu|.
     \end{aligned}
    \end{equation*}

Thus, it holds that  
    \begin{equation*}
    \begin{aligned}
     \abs{f_{i}^{\delta} - f_{i}} &\leq \abs{ \frac{u_{i}^{0,\delta} + u_{i}^{1,\delta}}{1 - 2x_{t.p.,i}^{\delta}} - \frac{u_{i}^{0} + u_{i}^{1}}{1 - 2x_{t.p.,i}}} + \abs{ \frac{u_{i}^{0} + u_{i}^{1}}{1 - 2x_{t.p.,i}} - \frac{u_{i}^{0} + u_{i}^{1}}{1 - 2x_{0,i}}} \\ 
     &\leq \frac{u_{i}^{0} + u_{i}^{1} + 1}{a^2}\cdot \delta + \abs{ u_{i}^{0} + u_{i}^{1}} \frac{C_x}{a^2} \mu| \ln \mu|,
     \end{aligned}
    \end{equation*}
which yields the required inequality with $C_a = \max( \frac{u_{i}^{0} + u_{i}^{1} + 1}{a^2}, \abs{ u_{i}^{0} + u_{i}^{1}} \frac{C_x}{a^2})$.
\end{proof}

	We then smooth the points $f_i^{\delta}$ by solving the minimization problem:
	\begin{equation}\label{fEpsilon}
		f^{\epsilon_1}(t) = \arg\min_{C^1(\mathscr{T})} \frac{1}{k+1} \sum_{i=0}^k \qty(v(t_i)-f_i^{\delta})^2 + \epsilon_1\norm{\displaystyle\pdv[2]{v}{t}}_{L^2(\mathscr{T})},
	\end{equation}
	where  $ \mathscr{T}=[0,T]$ and the regularisation parameter $\epsilon_1$ is chosen is chosen such that the minimum element \(f^{\epsilon_1}(t_i)\) in \eqref{fEpsilon} satisfies the equality:
	\begin{equation*}
		\frac{1}{k+1}\sum_{i=0}^k(f^{\epsilon_1}(t_i)-f_i^{\delta})^2=\delta^2.
	\end{equation*}

	\begin{lemma} 
    \label{lemma3} 
    Take $f$ as the exact function satisfying the original equation \eqref{InitProb} and let $f_0=\displaystyle\frac{u^0+u^1}{1-2x_{t.p.}}$. Then, subject to Assumption \ref{Cond2}, there exists a constant $C_3$, which remains unaffected by the variables $\mu$ and $t$, such that:
	\begin{equation}\label{FstarF0Est}
		\left\| f-f_0 \right\|_{L^p(\mathscr{T})}\leqslant C_3\mu| \ln \mu|, \quad {\forall} p\in (0,+\infty).
	\end{equation}
	\end{lemma}
	\begin{proof}
    From equation \eqref{x0}, together with Assumption \ref{Cond4} and the estimate \eqref{estim2}, we have:
	\begin{equation*}
		f = \frac{u^0+u^1}{1-2x_0}, \quad f_0 = \frac{u^0+u^1}{1-2x_{t.p.}}, \quad \left|{x_0 -x_{t.p.}}\right|_{L(\mathscr{T})} < C_x\mu| \ln \mu|.
	\end{equation*}

Using the bound within the integral
\begin{equation*}
\begin{aligned}
\left\|x_0-x_{\mathrm{t} . \mathrm{p} .}\right\|_{L^p(\mathscr{T} )}&=\left(\int_0^T\left|x_0(t)-x_{\mathrm{t.p.}}(t)\right|^p d t\right)^{1 / p} \\
&<\left(\int_0^T\left({C_x \mu| \ln \mu|}\right)^p d t\right)^{1 / p} = T^{1 / p} {C_x \mu| \ln \mu|},
\end{aligned}
\end{equation*}
we now estimate the \( L^p \)-norm of the difference between \( f \) and \( f_0 \):
        \begin{equation*}
        \begin{aligned}
				\norm{f-f_0}_{L^p(\mathscr{T})} &=  \norm{\frac{u^0+u^1}{1-2x_0}-\frac{u^0+u^1}{1-2x_{t.p.}}}_{L^p(\mathscr{T})} \\
				&\leqslant  \norm{x_0-x_{t.p.}}_{L^p(\mathscr{T})} \norm{\frac{2}{(1-2x_0)(1-2x_{t.p.})}} _{L^p(\mathscr{T})} \norm{u^0+u^1}_{L^p(\mathscr{T})} \\
                &< \norm{u^0+u^1}_{L^p(\mathscr{T})}\cdot \dfrac{C_{x}}{a^{2}}T^{2/p}\mu| \ln \mu|.
	     \end{aligned}
        \end{equation*}
	
	To obtain \eqref{FstarF0Est}, we use the estimate with $C_3=\norm{u^0+u^1}_{L^p(\mathscr{T})}\cdot \dfrac{C_{x}}{a^{2}}T^{2/p}$.
	\end{proof}
	
	Following the approach used in Lemma \ref{lemma2}, we write:
	
\begin{lemma} 
\label{lemma4} 
    Assume that $f(\cdot)\in C^2(L^2(\mathscr{T})) $ and let $f^{\epsilon_1}(\cdot)$ be the minimizer of problem \eqref{fEpsilon}. Then, there exists a constant $C_{4}$, independent of the parameters $\mu$ and $t$, such that:
	\begin{equation*}\label{uEps3}
        \begin{array}{ll}
		\norm{f^{\epsilon_1}(\cdot)-f(\cdot)}_{H^1(\mathscr{T})} &\leqslant 10\sqrt{2}\qty(h_1\norm{\displaystyle\pdv[2]{f(t)}{t}}_{L^2(\mathscr{T})} + \sqrt{\delta}\norm{\displaystyle \frac{\partial^2 f(t)}{\partial t^2}}_{L^2(\mathscr{T})}^{1/2})
         \\
        &\leqslant C_4(h_1+\sqrt{\delta}).
           \end{array}
	\end{equation*}
	\end{lemma}
    
	Using Lemmas \ref{lemma3} and \ref{lemma4} with standard arguments of approximation theory, it is possible to introduce the following Theorem:
	
	\begin{theorem} $f^{\delta}$, defined in \eqref{minProbFdelta}, is a stable approximation to the exact coefficient function $f$. In addition, there is an estimate for the rate of convergence:
	\begin{equation}
        \label{Thm4Estimate}
		\norm{f-f^{\delta}}_{L^2(\mathscr{T})} = \mathcal{O}(\mu| \ln \mu|+h_1+\sqrt{\delta}).
	\end{equation}
	
	Moreover, if $\mu=\mathcal{O}(\delta^{\varepsilon_1+1/2})$ ($\varepsilon_1$ is any positive constant) and $h_1=\mathcal{O}(\sqrt{\delta})$, then we have the following estimate:
	\begin{equation*}
		\norm{f-f^{\delta}}_{L^2(\mathscr{T})} = \mathcal{O}(\sqrt{\delta}).
	\end{equation*}
	\end{theorem}

    \begin{proof}
   The proof follows the same structure as the proof of Theorem \ref{theorem3}. By using the triangle inequality, we start with:
   \[
   \|f - f^{\delta}\|_{L^2(\mathscr{T})} \leq \|f - f_0\|_{L^2(\mathscr{T})} + \|f_0 - f^{\delta}\|_{L^2(\mathscr{T})}.
   \]
   
   For the first term, applying Lemma \ref{lemma3}, we have $\|f - f_0\|_{L^2(\mathscr{T})} \leq C_3 \mu| \ln \mu|$, while for the second term, from Lemma \ref{lemma4} we can derive $\|f_0 - f^{\delta}\|_{L^2(\mathscr{T})} \leq C_4(h_1 + \sqrt{\delta})$. Combining these results, we obtain the estimate \eqref{Thm4Estimate}.  
\end{proof}

\section{Simulations}
    \label{examples}

In this section, we explore the cases where the coefficient functions are $ F(x,t)=f(x)$ and $F(x,t)=f(t) $ for both the forward and inverse problems. Since explicit solutions for the considered non-linear PDEs are unavailable, we evaluate the accuracy of the obtained asymptotic solution (e.g. $U_0(x,t)$) by using a high-order numerical solution on fine grids as the reference solution (treated as the exact solution). The relative error of the asymptotic solution is then calculated on the same fine grids.

\subsection{First example with coefficient function $f(x)$}

	\subsubsection{Forward problem}
    
    For this example, we choose the coefficient function $f(x) = x^2 + 1.5$, boundary values $u^0 = -4, u^1 = 4.3$, and set the small parameter $\mu=0.02$.  This leads us to the following initial boundary value problem (the initial condition will be designed later):
	\begin{equation}\label{ExProb1}
		\begin{cases}
			\displaystyle \mu\qty(\pdv[2]{u}{x} - \pdv{u}{t}) = -u\pdv{u}{x} + f(x)u, \quad x \in (0,1),\, t\in [0,2],\\
			u(0,t) = -4,\quad u(1,t) = 4.3, \quad t\in [0,2],\\
			u(x,t) = u(x,t+T), \quad x \in (0,1), \, t\in [0,2].
		\end{cases}
	\end{equation}

To verify Assumption \ref{Cond1}, we consider:
$$\max_{\substack{x \in [0,1] \\  t \in \mathbb{R}_+}} \left( \int\limits_{0}^{x} (\tau^2 + 1.5) d\tau \right)=\max_{\substack{x \in [0,1] \\  t \in \mathbb{R}_+}} \left( \int\limits_{x}^{1} (\tau^2 + 1.5) d\tau \right)=1.833.$$

Next, we observe that $u^{0} = -4 < -1.833$, $u^{1} = 4.3 > 1.833$.
Additionally, the difference between $u^{1} $ and  $u^{0} $  satisfies: $u^{1}-u^{0}=8.3>2 \times (0.02)^{\theta} \ \text{for all} \ \theta \in (0,1)$.
Since all these conditions hold, Assumption \ref{Cond1} is satisfied.

The outer functions and the function $x_0(t)$, which defines the location of the zero-order transition layer, can be determined as follows:
	\begin{align*}
		&\varphi^l(x)=-4 + 1.5 x + 0.33333 x^3, 
		&\varphi^r(x)=2.46667 + 1.5 x + 0.33333 x^3,\\
		&x_0(t) \equiv 0.486 \in (0,1).
	\end{align*}
    
Next, we evaluate the derivative of \( I \) with respect to \( x \) at \( x_0 \):
$$\frac{d I}{dx}=\frac{d \left((\varphi^l)^2-(\varphi^r)^2 \right)}{ d x} (x_0) =-25.6856<0.$$ 

This confirms that Assumption \ref{Cond2} holds for the given functions and transition layer location.

The initial function is given by the equation:
\begin{flalign*}
		&u_{init}=
        \begin{cases}
			0.333 x^3+1.5 x-4+\frac{6.465}{ \exp\left(78.555-161.637x\right)+1}, & 0\leq x\leq 0.486, \\
 0.333 x^3+1.5 x+2.467-\frac{6.467}{ \exp\left(161.697 x-78.584\right)+1}, & 0.486\leq x\leq 1,
		\end{cases}
	\end{flalign*}   
which agrees with zero-order asymptotic solution at initial time $u_{init}=U_0(x,t=0)$. Therefore, Assumption \ref{Cond3} is satisfied.

  Based on the discussion above, Assumptions \ref{Cond1}-\ref{Cond3} are satisfied, and the asymptotic solution $ U_0(x,t)$ is demonstrated in Fig.$\eqref{my x1}$. We also illustrate a numerical solution $u(x,t)$ (by the finite volume method) of the problem \eqref{ExProb1} in Fig.$\eqref{my x2}$, which is also used for the generation of noisy data in the inverse problem.

    To analyze the effectiveness of the method we calculate relative error:
	\begin{equation*}
		\frac{\norm{U_0(x,t)-u(x,t)}_{L^2([0,1] \times [0,2])}}{\norm{u(x,t)}_{L^2([0,1] \times [0,2])}} = 0.010154.
	\end{equation*}

    \begin{figure}[H]
    \centering
    \begin{subfigure}[b]{0.4\textwidth}
        \centering
        \includegraphics[width=\textwidth]{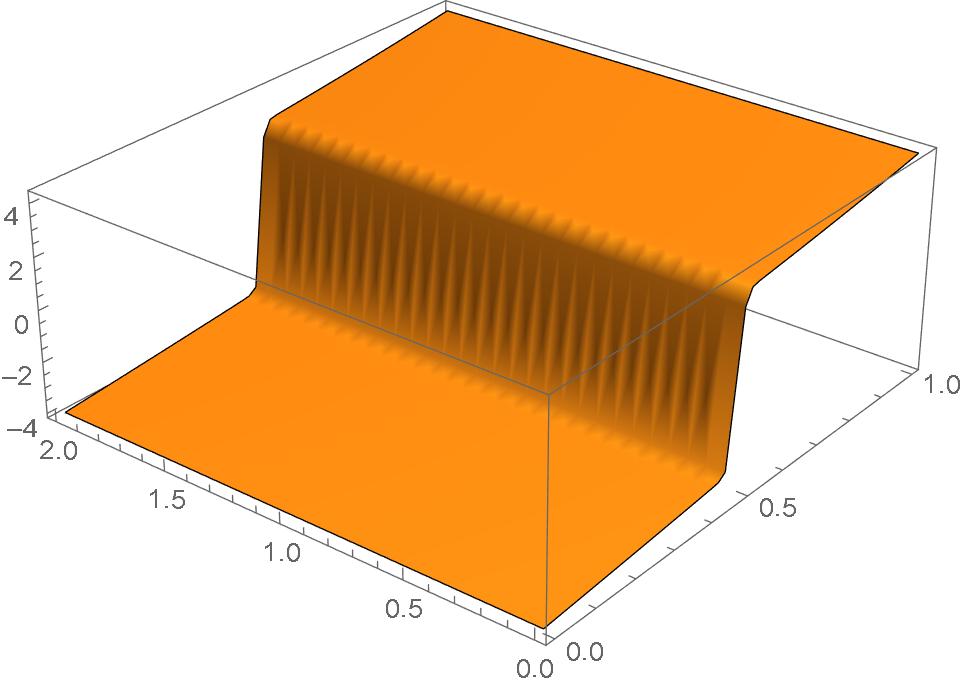}
        \caption{Asymptotic solution with $f(x)=x^2 + 1.5,x\in [0,1],t\in [0,2].$}
        \label{my x1}
    \end{subfigure}
    \hfill
    \begin{subfigure}[b]{0.4\textwidth}
        \centering
        \includegraphics[width=\textwidth]{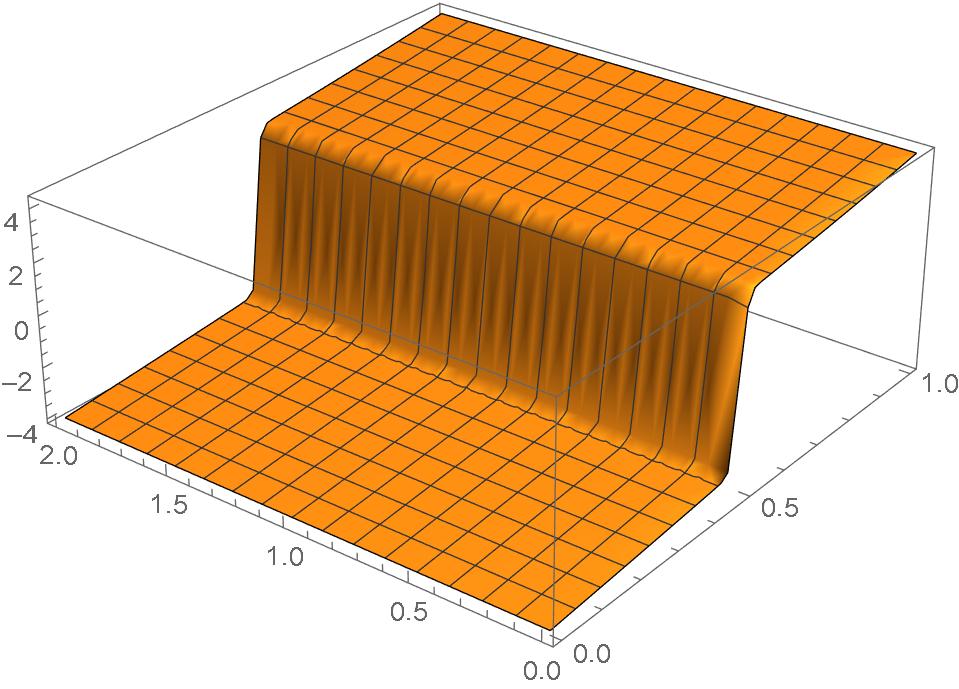}
        \caption{Numerical solution  with $f(x)=x^2 + 1.5,x\in [0,1],t\in [0,2].$}
        \label{my x2}
    \end{subfigure}
    \caption{Comparison of asymptotic and numerical solutions of equation \eqref{ExProb1}.}
\end{figure}

	\subsubsection{Inverse problem}

      We generate noisy data for $i = 0, \ldots, k$:
\begin{align}\label{noisyData}
	\omega_i^\delta := \left[1 + \delta(2 \cdot \text{rand} - 1)\right] \pdv{u(x_i,t_0)}{x},
\end{align}
	where $\delta$ is the noise level and ``rand'' returns pseudorandom values drawn from the uniform distribution [0, 1]. 
    
	
	In our numerical experiment, we use a noise level of $\delta = 0.1 \% $ or $1 \% $. The parameters are set as $t_0 = 1$, $k^l = 90$, $k^r = 110$, and $k = 200$. The grid values $\displaystyle\pdv{u(x_i,t_0)}{x}$ are taken from the forward problem.
	
	We operate only with nodes located in two areas on both sides of the transition layer. Therefore, indices $i = 0,\ldots,k^l$ and $i = k^r,\ldots,k$ are introduced. According to standard procedure, we add a uniform noise \eqref{noisyData} to our values $\displaystyle\pdv{u(x_i,t_0)}{x}$. This step allows us to determine noisy data $\{\omega^\delta_i \}^{k^l}_{i=0}$ and $\{\omega^\delta_i \}^{k^r}_{i=0}$ in their corresponding areas.

    	From the optimization problem \eqref{minProbFdelta} we derive an approximate source function $f^\delta(x)$.
	Its graphical representation is presented through Fig.\eqref{my f(x)} and Fig.\eqref{my f(x)_1}.
	
	Lastly, the relative errors of the source function with two noisy data sets of different levels were calculated:
	
	for $\delta=0.1\%$: $\displaystyle\frac{\left\| f^{\delta}-f\right\|_{ L^2(0,1)}}{\left\| f \right\|_{L^2(0,1)}}  = 0.018169$,

	for $\delta=1\%$: $\displaystyle\frac{\left\| f^{\delta}-f\right\|_{ L^2(0,1)}}{\left\| f \right\|_{L^2(0,1)}}=0.027691$.
    
    \begin{figure}[H]
    \centering
    \begin{subfigure}[b]{0.4\textwidth}
        \centering
        \includegraphics[width=\textwidth]{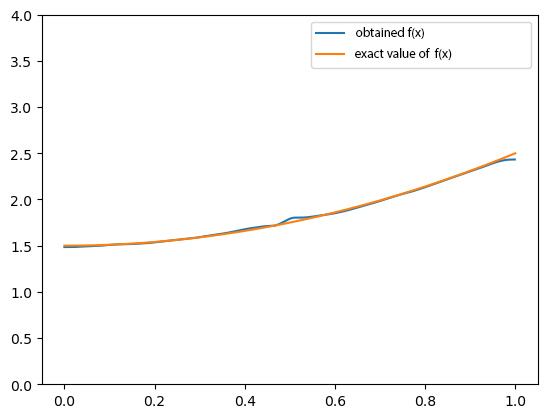}
        \caption{Reconstructed coefficient function $f^\delta (x)$, where $\delta=0.1\%$, compared to the exact original function $f(x)=x^2 + 1.5,x\in [0,1],\mu=0.02.$}
        \label{my f(x)}
    \end{subfigure}
    \hfill
    \begin{subfigure}[b]{0.4\textwidth}
        \centering
        \includegraphics[width=\textwidth]{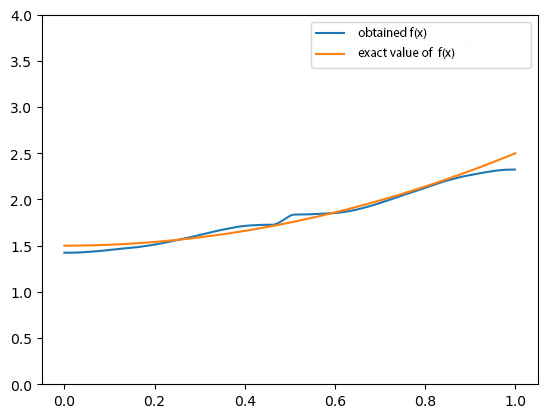}
        \caption{Reconstructed coefficient function $f^\delta(x)$, where $\delta=1\%$, compared to the exact original function $f(x)=x^2 + 1.5,x\in [0,1],\mu=0.02.$}
        \label{my f(x)_1}
    \end{subfigure}
    \caption{Comparison of reconstructed coefficient functions for different $\delta$ values.}
\end{figure}

	\subsection{Second example with coefficient function $f(t)$}
	\subsubsection{Forward problem}
	In this example, we consider problem \eqref{InitProb} with a given coefficient function $f(t) = \cos(t) + 2$, left and right boundary functions $u^{0}(t) = -4-0.5\sin(t)$, $u^{1}(t) = 5+0.7\cos(t)$, and small parameter $\mu=0.02$. I.e. we consider the following initial boundary value problem:
	\begin{equation}\label{ExProb2}
		\begin{cases}
			\displaystyle \mu\qty(\pdv[2]{u}{x} - \pdv{u}{t}) =	-u\pdv{u}{x} + f(t)u, \quad x \in (0,1),\, t\in [0,4\pi],\\
			u(0,t)=-4-0.5\sin(t),\quad u(1,t)=5+0.7\cos(t),\quad t\in [0,4\pi],\\
			u(x,t)=u(x,t+T),\quad x \in (0,1),\, t\in [0,4\pi].
		\end{cases}
	\end{equation}

This example also satisfies all the assumptions of our approach. To demonstrate this, we verify each assumption as follows. First, we note that 
$$\max_{\substack{x \in [0,1] \\ t \in [0,4\pi]}} \left( \int\limits_{0}^{x} (\cos(t) + 2) \, \mathrm{d}\tau \right) =\max_{\substack{x \in [0,1] \\ t \in [0,4\pi]}} \left( \int\limits_{x}^{1} (\cos(t) + 2) \, \mathrm{d}\tau  \right) =3.$$

Thus, Assumption \ref{Cond1} is satisfied due to the following estimates.
 
1. $\displaystyle  \max_{t \in [0,4\pi]} u^0(t) = 
\max_{t \in [0,4\pi]} \left( -4-0.5\sin(t) \right)=-4<-3$.

2. $\displaystyle  \max_{t \in [0,4\pi]} u^1(t) =  \min_{t \in [0,4\pi]} \left(  5+0.7\cos(t) \right)=5>3$.  

3. Since \(u^1(t) - u^0(t) = 9 + 0.7\cos(t) + 0.5\sin(t) > 8.13 > 0.04 \ge 2\mu^\theta \) for all \(\theta \in (0,1)\).

On the other hand, for this model problem, we can derive the exact outer functions of zero order, as well as an equation for the middle line of the inner layer, namely
	\begin{align}
		&\varphi^l(x,t)=-4 + 2 x + x \cos(t) - 0.5 \sin(t), \label{Ex2varphiL} \\
		&\varphi^r(x,t)=3 + 2 x - 0.3\cos(t) + x \cos(t), \label{Ex2varphiR} \\
		&x_0(t)=\displaystyle\frac{0.5 + 0.15\cos(t) + 0.25\sin(t)}{2 + \cos(t)}. \label{x0texample2}
	\end{align}

\begin{lemma} 
\label{lemmaEx2} 
For all $t\in [0,4\pi]$, $x_0(t)\in(0,0.5)$, where $x_0(t)$ is defined in \eqref{x0texample2}.
\end{lemma}

\begin{proof}
Since the numerator and denominator in \eqref{x0texample2} are always positive, it is sufficient to prove \( x_0(t) < 0.5 \). Indeed, according to 
\begin{align*}
0.5+0.35\cos(t)-0.25\sin(t)&=0.5-\sqrt{0.35^2+0.25^2}\cos{\left(t-\arctan(-\frac{5}{7})\right)}\\
&\geq 0.5-\sqrt{0.35^2+0.25^2}\approx 0.069 >0,
\end{align*}
we derive $1+0.5\cos(t) > 0.5 + 0.15\cos(t)+0.25\sin(t)$, which yields the required uniform upper bound for $x_0(t)$. 
\end{proof}

From the derivative of \( I \) with respect to \( x \) at \( x_0 \), we derive that  
\[
\frac{dI}{dx} = \frac{d \left( (\varphi^l)^2 - (\varphi^r)^2 \right)}{dx}(x_0(t)) = 2(2 + \cos(t))(\varphi^l - \varphi^r).
\]

By the explicit formulas of $\varphi^{l,r}$ in \eqref{Ex2varphiL}-\eqref{Ex2varphiR}, we have \( (\varphi^l - \varphi^r) < 0 \) for all \( x \in (0,1) \) and \( t \in [0,4\pi] \), which implies that  
\begin{equation}
\label{checkassumtion}
\frac{dI}{dx} < 0.
\end{equation}
According to Lemma \ref{lemmaEx2} and inequality \eqref{checkassumtion}, Assumption \ref{Cond2} is satisfied for the model problem \eqref{ExProb2}.

In this example, the initial function is given by the equation:
	\begin{flalign*}
		u_{init}=
        \begin{cases}
3 x-4 +\frac{6.7}{ \exp(36.292 - 167.5x)+1}, & 0\leq x\leq \frac{13}{60}, \\
 3 x+2.7 -\frac{6.7}{ \exp(167.5x - 36.292)+1}, & \frac{13}{60}\leq x\leq 1,
		\end{cases}
	\end{flalign*} 
which agrees with zero-order asymptotic solution at initial time $u_{init}=U_0(x,t=0)$. Therefore, Assumption \ref{Cond3} is satisfied.

     Thus, all the Assumption \ref{Cond1}-\ref{Cond3} are met, and the asymptotic solution $ U_0(x,t)$ is illustrated in Fig.$\eqref{my x4}$. We also draw a numerical solution $ u(x,t)$ (by the finite volume method) of the problem \eqref{ExProb2} in Fig.$\eqref{my x5}$, which is also used for the generation of noisy data in the inverse problem. The relative error is reported below. 
	\begin{equation*}
		\frac{\norm{U_0(x,t)-u(x,t)}_{L^2([0,1]\times[0,4\pi])}}{\norm{u(x,t)}_{L^2([0,1]\times[0,4\pi])}} = 0.0176743.
	\end{equation*}

    \begin{figure}[H]
    \centering
    \begin{subfigure}[b]{0.4\textwidth}
        \centering
        \includegraphics[width=\textwidth]{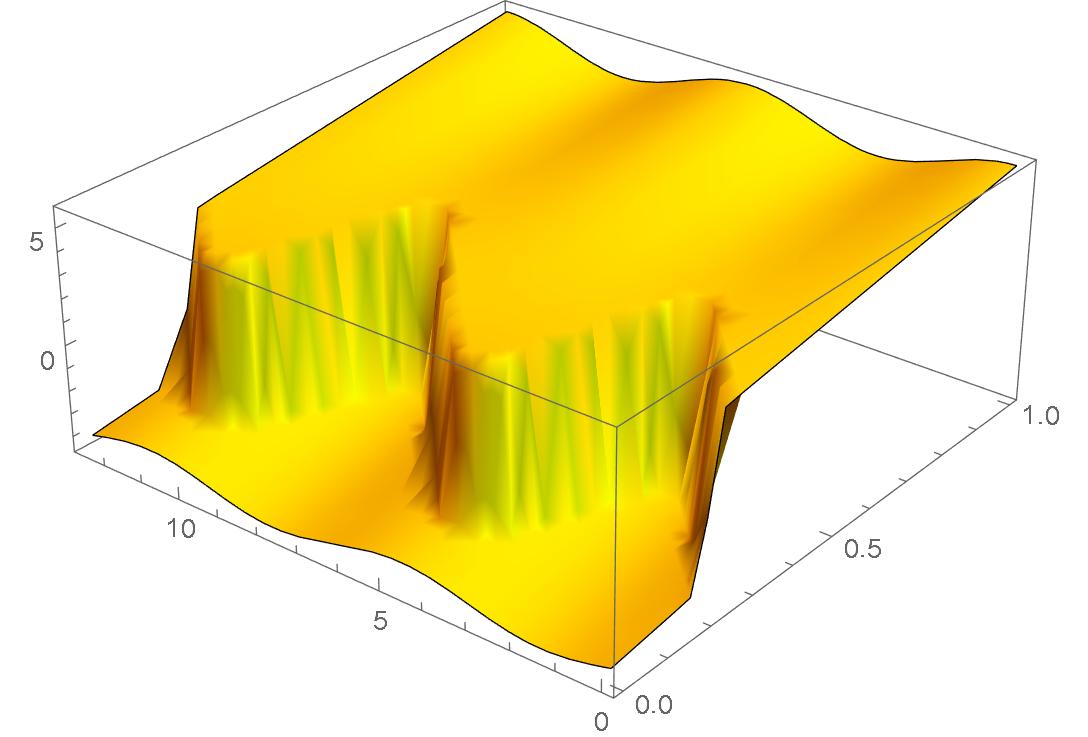}
        \caption{Asymptotic solution with $f(t)=2+\cos(t), x \in [0,1], t \in [0,4\pi].$}
        \label{my x4}
    \end{subfigure}
    \hfill
    \begin{subfigure}[b]{0.4\textwidth}
        \centering
        \includegraphics[width=\textwidth]{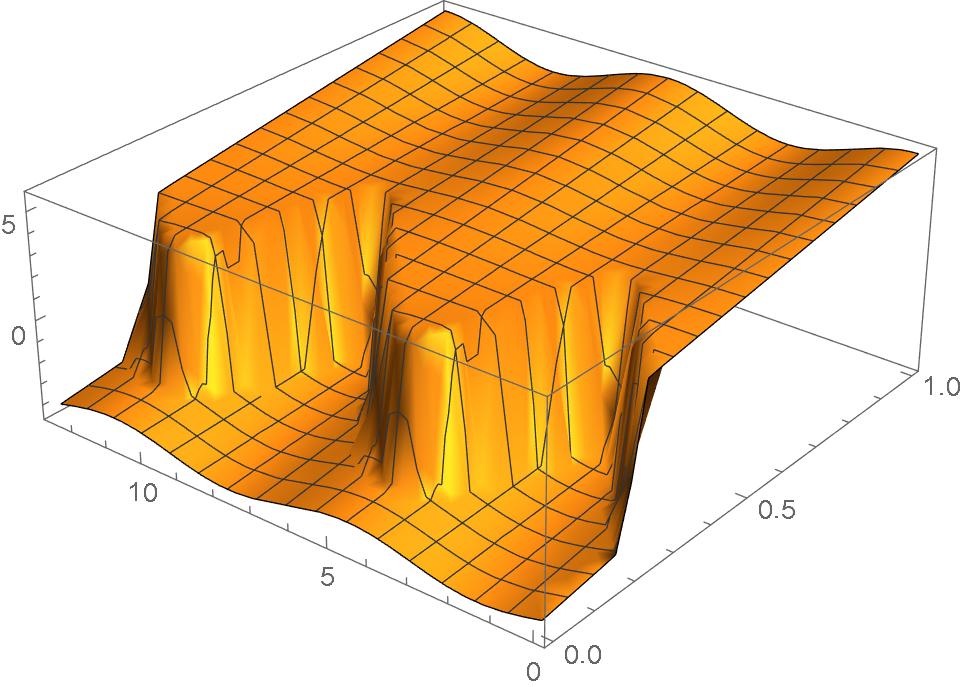}
        \caption{Numerical solution with $f(t)=2+\cos(t), x \in [0,1], t \in [0,4\pi].$}
        \label{my x5}
    \end{subfigure}
    \caption{Comparison of asymptotic and numerical solutions of equation \eqref{ExProb2}.}
\end{figure}
	
	\subsubsection{Inverse problem}
	
	Now we consider model \eqref{ExProb2} in the case when the source function is $f(t)$. Moreover, we assume that only the values of $\{u^0(t_i),u^1(t_i),x_{t.p.}(t_i)\}$ are known.

    The function $x_{t.p.}(t)$, which defines the location of the inner layer for the numerical solution, is obtained by numerically solving the equation \eqref{barUatTP}, i.e. we find an approximate solution of $x_{t.p.}(t)$ by solving the following nonlinear equation. 
    \begin{equation*}\label{obtainxtp}
		u(x_{t.p.}(t),t) = -0.5 + 2 x_{t.p.}(t) + x_{t.p.}(t) \cos(t) - 0.15 \cos(t) - 0.25 \sin(t) .
	\end{equation*}
    
The resulting function is shown in Figure \ref{fig:xtp}. Numerically, we find that $ x_{t.p.}(t) >0$, and the maximum value of the function \( x_{t.p.}(t) \) is approximately $\max\limits_{t \in [0, 4\pi]} x_{t.p.}(t) \approx 0.442$. Therefore, there exists a value \( a = 0.05 \) that satisfies Assumption \ref{Cond4}.

\begin{figure}[!htb]
\begin{center}
\includegraphics[width=0.4\linewidth]{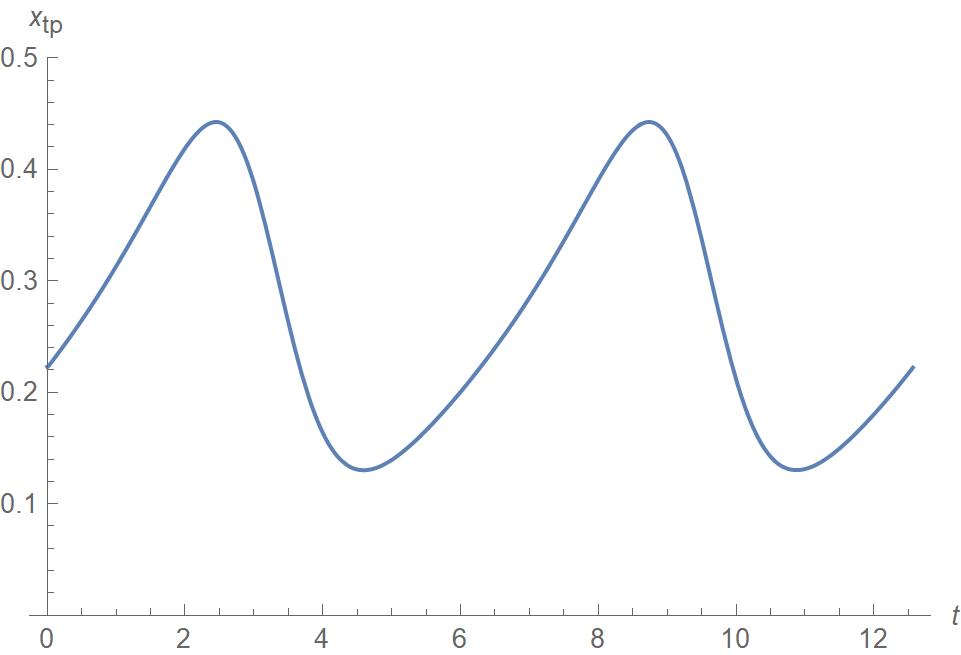}
\caption{Function $x_{t.p.}$ numerically obtained from equation \eqref{barUatTP}, representing the location of the transition layer.}
\label{fig:xtp}
\end{center}
\end{figure}

    According to the previous example, we synthesize measurement data with indexes $i = 0, \ldots, k$ from the direct problem \eqref{ExProb2} by using the numerical method (see Fig.\eqref{my x5}):
\begin{align*}
	&u_i^{0,\delta} := \left[1 + \delta(2 \cdot \text{rand} - 1)\right] u^{0}(t_i),\\
	&u_i^{1,\delta} := \left[1 + \delta(2 \cdot \text{rand} - 1)\right] u^{1}(t_i),\\
	&x_{{t.p.},i}^\delta := \left[1 + \delta(2 \cdot \text{rand} - 1)\right] x_{t.p.}(t_i).
\end{align*}
	

	
	During calculations $\delta = 0.1\%$ or $1\%$ and $k = 40$ were used. We constructed the smoothed data $f^{\epsilon_1}(t)$ according to the optimization problem \eqref{fEpsilon}:
	\begin{equation*}\label{f1epsilon}
		f^{\epsilon_1}(t) = \mathop{\arg\min}_{C^1(0,4\pi)}\frac{1}{k+1} \sum_{k=0}^k(v(t_i)-f_i^{\delta})^2 + \epsilon_1\norm{\displaystyle\pdv[2]{v}{t}}_{L^2(0,4\pi)},
	\end{equation*}
	where the points $f_i^{\delta}$ was obtained from equation \eqref{restoreeqft}. 
    
    The reconstructed time-dependent functions $f^{\epsilon_1}(t)$ for different noise levels are displayed in Figures \ref{myF(t)}.  The relative errors of the $f^{\epsilon_1}$ give us values:
	
	for $\delta=0.1\%$: $\displaystyle\frac{\left\| f^{\epsilon_1}-f\right\|_{ L^2(0,4\pi)}}{\left\| f \right\|_{L^2(0,4\pi)}}=0.018942$,
	
	for $\delta=1\%$: $\displaystyle\frac{\left\| f^{\epsilon_1}-f\right\|_{ L^2(0,4\pi)}}{\left\| f \right\|_{L^2(0,4\pi)}}=0.039787$.
	
\begin{figure}[!tbh]
    \centering
    \begin{subfigure}[b]{0.4\textwidth}
        \centering
        \includegraphics[width=\textwidth]{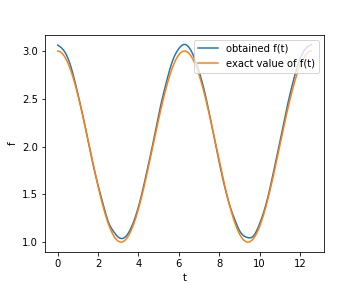}
        \caption{Reconstructed coefficient function $f^{\epsilon_1} (t)$, for $\delta = 0.1\%$ compared to the exact original function $f(t)=\cos(t)+2,  t \in [0,4\pi], \mu=0.02.$}
    \end{subfigure}
    \hfill
    \begin{subfigure}[b]{0.4\textwidth}
        \centering
        \includegraphics[width=\textwidth]{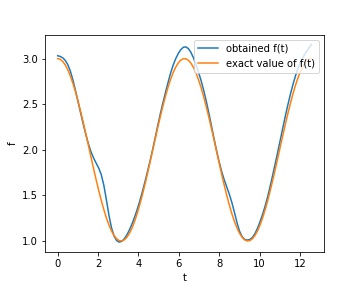}
        \caption{Reconstructed coefficient function $f^{\epsilon_1} (t)$, for $\delta = 1\%$ compared to the exact original function $f(t)=\cos(t)+2,  t \in [0,4\pi], \mu=0.02.$}
    \end{subfigure}
    \caption{Comparison of reconstructed coefficient functions for different $\delta$ values.}
\label{myF(t)}
\end{figure}

\section{Conclusion}
\label{conclusion}

This study presents a comprehensive analysis of singularly perturbed reaction-diff usion-advection equations, with a focus on solutions exhibiting moving fronts. Using asymptotic expansion methods, we constructed approximate solutions for non-stationary differential equations with time-dependent right-hand sides. Building on this framework, we extended an asymptotic-based algorithm to solve inverse problems aimed at determining spatially and temporally dependent coefficient functions, establishing key conditions to ensure the existence of solutions. Additionally, we derived convergence rate estimates for the inverse problem and, through asymptotic analysis, determined the initial conditions necessary for numerically solving the main problem. The effectiveness of the proposed approach was demonstrated through two groups of numerical experiments, highlighting its precision in reconstructing coefficient functions under various noise levels. This research contributes to the deeper understanding of reaction-diffusion-advection equations and has significant implications for applications in environmental modeling, medical imaging, and materials science.


\end{fulltext}
	\bibliographystyle{unsrt} 
        \bibliography{ref}

\begin{thebibliography}{10}

\bibitem{Cos2014}
C.~Cosner.
\newblock Reaction-diffusion-advection models for the effects and evolution of
  dispersal.
\newblock {\em Discrete and Continuous Dynamical Systems}, 34(5):1701--1745,
  May 2014.

\bibitem{ComNegBrau2023}
A.~Comolli, L.~Negrojevic, F.~Brau, and A.~De~Wit.
\newblock Effect of radial advection on autocatalytic reaction-diffusion
  fronts.
\newblock {\em Physical Chemistry Chemical Physics}, 25:10604--10619, March
  2023.

\bibitem{ZhaForGul2016}
Y.~Zhang, G.L. Lin, P.~Forssen, M.~Gulliksson, T.~Fornstedt, and X.L. Cheng.
\newblock A regularization method for the reconstruction of adsorption
  isotherms in liquid chromatography.
\newblock {\em Inverse Problems}, 32(10):105005, August 2016.

\bibitem{ZhaLinGul2017}
Y.~Zhang, G.~Lin, M.~Gulliksson, P.~Forssen, T.~Fornstedt, and X.~Cheng.
\newblock An adjoint method in inverse problems of chromatography.
\newblock {\em Inverse Problems in Science and Engineering}, 25(8):1112--1137,
  August 2016.

\bibitem{Emken2020}
N.~Emken and C.~Engwer.
\newblock A reaction-diffusion-advection model for the establishment and
  maintenance of transport-mediated polarity and symmetry breaking.
\newblock {\em Frontiers in Applied Mathematics and Statistics}, 6, November
  2020.

\bibitem{Erofeev2021}
V.I. Erofeev and A.V. Leonteva.
\newblock Plane longitudinal waves in a fluid-saturated porous medium with a
  nonlinear relationship between deformations and displacements of the liquid
  phase.
\newblock {\em Computational Continuum Mechanics}, 14(1):5--11, June 2021.

\bibitem{ManPim2020}
N.~Manitcharoen and B.~Pimpunchat.
\newblock Analytical and {{Numerical Solutions}} of {{Pollution Concentration}}
  with {{Uniformly}} and {{Exponentially Increasing Forms}} of {{Sources}}.
\newblock {\em Journal of Applied Mathematics}, 2020:e9504835, April 2020.

\bibitem{Nefedov2013}
N.N. Nefedov, L.~Recke, and K.R. Schneider.
\newblock Existence and asymptotic stability of periodic solutions with an
  interior layer of reaction-advection-diffusion equations.
\newblock {\em Journal of Mathematical Analysis and Applications},
  405(1):90--103, September 2013.

\bibitem{Levashova2018}
N.T. Levashova, N.N. Nefedov, and A.V. Yagremtsev.
\newblock Existence of a solution in the form of a moving front of a
  reaction-diffusion-advection problem in the case of balanced advection.
\newblock {\em Izvestiya: Mathematics}, 82(5):984--1005, October 2018.

\bibitem{amp20221012}
M.~M. Rahaman, H.~Takia, Md.~K. Hasan, Md.~B. Hossain, S.~Mia, and K.~Hossen.
\newblock Application of advection diffusion equation for determination of
  contaminants in aqueous solution: A mathematical analysis.
\newblock {\em Applied Mathematics and Physics}, 10(1):24--31, March 2022.

\bibitem{Liang2010}
J.~Liang, G.~Zeng, S.~Guo, A.~Wei, X.~Li, L.~Shi, and C.~Du.
\newblock Optimal solute transport in heterogeneous aquifer: coupled inverse
  modelling.
\newblock {\em International Journal of Environment and Pollution},
  42(1/2/3):258, August 2010.

\bibitem{Got1965}
B.S. Gottfried.
\newblock {A Mathematical Model of Thermal Oil Recovery in Linear Systems}.
\newblock {\em Society of Petroleum Engineers Journal}, 5(03):196--210,
  September 1965.

\bibitem{Liu2023}
T.~Liu, D.~Ouyang, L.~Guo, R.~Qiu, Y.~Qi, W.~Xie, Q.~Ma, and C.~Liu.
\newblock Combination of multigrid with constraint data for inverse problem of
  nonlinear diffusion equation.
\newblock {\em Mathematics}, 11(13):2887, June 2023.

\bibitem{Tilioua2012}
A.~Tilioua, L.~Libessart, A.~Joulin, S.~Lassue, B.~Monod, and G.~Jeandel.
\newblock Determination of physical properties of fibrous thermal insulation.
\newblock {\em EPJ Web of Conferences}, 33:02009, October 2012.

\bibitem{Isa1990}
V.~Isakov.
\newblock {\em Inverse Source Problems}.
\newblock Mathematical Surveys and Monographs. American Mathematical Society,
  1990.

\bibitem{ChaZha2022}
D.~Chaikovskii and Y.~Zhang.
\newblock Convergence analysis for forward and inverse problems in singularly
  perturbed time-dependent reaction-advection-diffusion equations.
\newblock {\em Journal of Computational Physics}, 470:111609, June 2022.

\bibitem{Chaikovskii2D2023}
D.~Chaikovskii, A.~Liubavin, and Y.~Zhang.
\newblock Asymptotic expansion regularization for inverse source problems in
  two-dimensional singularly perturbed nonlinear parabolic pdes.
\newblock {\em CSIAM Transactions on Applied Mathematics}, 4(4):721--757, June
  2023.

\bibitem{Chaikovskii3D2023}
D.~Chaikovskii and Y.~Zhang.
\newblock Solving forward and inverse problems involving a nonlinear
  three-dimensional partial differential equation via asymptotic expansions.
\newblock {\em IMA Journal of Applied Mathematics}, 88(4):525--557, August
  2023.

\bibitem{Lukyanenko2019}
D.V. Lukyanenko, V.B. Grigorev, V.T. Volkov, and M.A. Shishlenin.
\newblock Solving of the coefficient inverse problem for a nonlinear singularly
  perturbed two-dimensional reaction-diffusion equation with the location of
  moving front data.
\newblock {\em Computers \& Mathematics with Applications}, 77(5):1245--1254,
  March 2019.

\bibitem{Lukyanenko2019v2}
D.V. Lukyanenko, M.A. Shishlenin, and V.T. Volkov.
\newblock Asymptotic analysis of solving an inverse boundary value problem for
  a nonlinear singularly perturbed time-periodic reaction-diffusion-advection
  equation.
\newblock {\em Journal of Inverse and Ill-posed Problems}, 27(5):745--758, July
  2019.

\bibitem{Lukyanenko2020}
D.V. Lukyanenko, I.V. Prigorniy, and M.A. Shishlenin.
\newblock Some features of solving an inverse backward problem for a
  generalized burgers equation.
\newblock {\em Journal of Inverse and Ill-posed Problems}, 28(5):641--649,
  September 2020.

\bibitem{Lukyanenko2021}
D.~Lukyanenko, T.~Yeleskina, I.~Prigorniy, T.~Isaev, A.~Borzunov, and
  M.~Shishlenin.
\newblock Inverse problem of recovering the initial condition for a nonlinear
  equation of the reaction-diffusion-advection type by data given on the
  position of a reaction front with a time delay.
\newblock {\em Mathematics}, 9(4):342, February 2021.

\bibitem{Lukyanenko2021v2}
D.V. Lukyanenko, A.A. Borzunov, and M.A. Shishlenin.
\newblock Solving coefficient inverse problems for nonlinear singularly
  perturbed equations of the reaction-diffusion-advection type with data on the
  position of a reaction front.
\newblock {\em Communications in Nonlinear Science and Numerical Simulation},
  99:105824, August 2021.

\bibitem{Shishatskii1988}
S.P. Shishat'skii.
\newblock Interior problems for elliptic and parabolic equations and carlemans
  estimates with distributive singularities.
\newblock {\em Soviet mathematics - doklady}, 36(3):516--518, February 1988.

\bibitem{Bukhgeim1993}
A.L. Bukhgeim.
\newblock Extension of solutions of elliptic equations from discrete sets.
\newblock {\em Journal of Inverse and Ill-posed Problems}, 1(1):17--32, January
  1993.

\bibitem{Tataru2004}
D.~Tataru.
\newblock Unique continuation problems for partial differential equations.
\newblock In Christopher~B. Croke, Michael~S. Vogelius, Gunther Uhlmann, and
  Irena Lasiecka, editors, {\em Geometric Methods in Inverse Problems and PDE
  Control}, pages 239--255, New York, NY, 2004. Springer New York.

\bibitem{Antipov2014}
N.N.~Nefedov E.A.~Antipov, N.T.~Levashova.
\newblock Asymptotics of the front motion in the reaction-diffusion-advection
  problem.
\newblock {\em Computational Mathematics and Mathematical Physics},
  54(10):1536--1549, October 2014.

\bibitem{Hess1991-ha}
P.~Hess.
\newblock {\em Periodic-parabolic boundary value problems and positivity}.
\newblock Pitman Research Notes in Mathematics Series. John Wiley \& Sons,
  Nashville, TN, March 1991.

\end{thebibliography}

\section{Appendix: Proofs of Theorems \ref{existenceTheorem} and \ref{asymptEstimates}}

\begin{proof}[Proof of Theorem \ref{existenceTheorem} (Existence and asymptotic solution)]

We construct upper and lower solutions $\beta(x,t,\mu)$ and $\alpha(x,t,\mu)$ for the problem \eqref{InitProb}. These functions must satisfy the following conditions:

\begin{enumerate}
  \item[\textbf{(C1):}] Ordering condition, ensure that the following inequality holds:
  \[
  \alpha(x, t, \mu) < \beta(x, t, \mu), \quad \forall (x, t) \in \overline{\Omega}, \quad \text{where } \overline{\Omega} := \{(x,t) :  x \in [0,1], t \in \mathbb{R}_+\}.
  \]

  \item[\textbf{(C2):}]  The differential operators \( L[\beta] \) and \( L[\alpha] \) should satisfy:
  \[
 L[\beta]:= \mu \left( \frac{\partial^2 \beta}{\partial x^2} - \frac{\partial \beta}{\partial t} \right) - \left( \beta \frac{\partial \beta}{\partial x} + F(x, t) \beta \right) \leq 0,
  \]
  and
  \[
  L[\alpha]:=  \mu \left( \frac{\partial^2 \alpha}{\partial x^2} - \frac{\partial \alpha}{\partial t} \right) - \left( \alpha \frac{\partial \alpha}{\partial x} + F(x, t) \alpha \right) \geq 0,
  \]
  for all \((x, t) \in \overline{\Omega}\).

  \item[\textbf{(C3):}] At the boundary points for \( t \in \mathbb{R}_+ \), the following inequalities must hold:
  \[
  \alpha(0, t, \mu) \leq u^{0}(t) \leq \beta(0, t, \mu),
  \]
  and
  \[
  \alpha(1, t, \mu) \leq u^{1}(t) \leq \beta(1, t, \mu).
  \]

  \item[\textbf{(C4):}] Continuity conditions for non-smooth solutions:
  
  \begin{itemize}
    \item If \(\beta\) is not smooth at some \( x = x_{\beta}(t) \) within \( (0, 1) \), then:
    \[
    \left. \frac{\partial \beta}{\partial x} \right|_{x = x_{\beta}(t)-0} - \left. \frac{\partial \beta}{\partial x} \right|_{x = x_{\beta}(t)+0} \geq 0.
    \]
    
    \item Similarly, if \(\alpha\) is not smooth at some \( x = x_{\alpha}(t) \) within \( (0, 1) \), then:
    \[
    \left. \frac{\partial \alpha}{\partial x} \right|_{x = x_{\alpha}(t)-0} - \left. \frac{\partial \alpha}{\partial x} \right|_{x = x_{\alpha}(t)+0} \leq 0.
    \]
  \end{itemize}
\end{enumerate}

We construct the upper solution as follows:

\begin{align*}
\beta^{l}(x, t, \mu) &= U_{n+1}^{l}\Big|_{\xi_{\beta}} + \mu^{n+1}\left[\epsilon^{l}(x,t) + q_{0}^{l}(\xi_{\beta}, t) + \mu q_{1}^{l}(\xi_{\beta}, t)\right], \\
&\qquad\qquad\qquad\qquad \text{for } 0 \leq x \leq x_{\beta}, \, \xi_{\beta} \leq 0, \, t \in \mathbb{R}_+ , \\
\beta^{r}(x, t, \mu) &= U_{n+1}^{r}\Big|_{\xi_{\beta}} + \mu^{n+1}\left[\epsilon^{r}(x,t) + q_{0}^{r}(\xi_{\beta}, t) + \mu q_{1}^{r}(\xi_{\beta}, t)\right], \\
&\qquad\qquad\qquad\qquad \text{for } x_{\beta} \leq x \leq 1, \, \xi_{\beta} \geq 0, \, t \in \mathbb{R}_+.
\end{align*}

Here, \(\xi_{\beta}\) is given by:
\[
\xi_{\beta} = \frac{x - x_{\beta}(t)}{\mu},
\]
where \(x_{\beta}(t) = X_{n+1}(t) - \mu^{n+1} \delta(t)\), with \(\delta(t)\) being a positive function to be determined later.

The functions $\epsilon^{l,r}(x,t)$ are defined as solutions to the following boundary value problems:

\[
\begin{cases}
\epsilon^{l,r}(x,t) \left( \frac{\partial \varphi^{l,r}(x, t)}{\partial x} -f(x, t)   \right) +  \varphi^{l,r}(x, t)  \frac{d \epsilon^{l,r}(x,t)}{d x}= R(t) , \\
\epsilon^{l}(0,t) = R^{l}(t), \quad \epsilon^{r}(1,t) = R^{r}(t).
\end{cases}
\]
where functions $R,R^{l},R^{r}$ have positive values for $ t \in \mathbb{R}_+ $.

The functions $\epsilon^{l,r}(x,t)$ have explicit solutions:

\begin{multline*} 
\epsilon^{l}= e^{\left(\int_x^1 \frac{\frac{\partial \varphi^{l} (\tau ,t)}{\partial \tau}-f(\tau ,t)}{\varphi^{l}(\tau ,t)} \, d\tau \right)} \\ 
\cdot \left( R^{l}(t) e^{ \left(\int_1^0 \frac{\frac{\partial \varphi^{l} (\tau ,t)}{\partial \tau}-f(\tau ,t)}{ \varphi^{l} (\tau ,t)} \,
   d\tau \right)} +\int_0^x \frac{R(t) e^{ \left(\int_1^{\zeta} \frac{\frac{\partial \varphi^{l} (\tau ,t)}{\partial \tau}-f(\tau ,t)}{ \varphi^{l} (\tau ,t)} \, d\tau \right)}}{\varphi^{l} (\zeta,t)} \, d\zeta \right),
\end{multline*}

\begin{multline*} \epsilon^{r}=e^{\left(\int_x^1 \frac{\frac{\partial \varphi^{r} (\tau ,t)}{\partial \tau}-f(\tau,t)}{\varphi^{r} (\tau,t)} \, d\tau\right)} \left(R^{r}(t)-\int_x^1 \frac{R(t) e^{\left(\int_1^{\zeta}
   \frac{\frac{\partial \varphi^{r} (\tau ,t)}{\partial \tau}-f(\tau,t)}{\varphi^{r}(\tau,t)} \, d\tau\right)}}{\varphi^{r}(\zeta,t)} \, d\zeta\right).
\end{multline*}
The functions $\epsilon^{l,r}(x,t)$ are positive for big enough $R^{l}(t),R^{r}(t)$.

The functions $q_{0}^{l,r}\left(\xi_{\beta}, t\right)$ are solutions of:
\begin{equation}\label{q0equation}
\begin{cases}
 \frac{\partial^2 q_0^{l,r}}{\partial \xi_\beta^2}
+  \left( \varphi^{l,r}(x_0(t), t) + Q_0^{l,r}  \right) \frac{\partial q_0^{l,r}}{\partial \xi_\beta} 
+   q_0^{l,r} \frac{\partial Q_0^{l,r}}{\partial \xi_\beta}
= \left( \frac{\partial \varphi^{l,r}}{\partial x_0}  \delta(t) -\epsilon^{l,r}(x_0(t),t) \right) \frac{\partial Q_0^{l,r}}{\partial \xi_\beta}, \\
q_{0}^{l,r}(0, t)  = - \epsilon^{l,r}\left(x_{\beta}\right)-\delta(t)\frac{\partial \varphi^{l,r}(x_0(t), t)}{\partial x_0} , \quad q_{0}^{l,r}(\mp \infty, t) = 0.
\end{cases}
\end{equation}

The functions $q_{0}^{l,r}\left(\xi_{\beta}, t\right)$ can be written explicitly:
\begin{multline} \label{q0explisit}
q_{0}^{l,r} (\overline{\xi}, t)=z^{l,r}(\overline{\xi},t) \bigg( q_{0}^{l,r}(0, t)   \\
 +\int_{0}^{\overline{\xi}} \frac{1}{z^{l,r}(s,t)} \int_{\mp \infty}^{s} \left( \frac{\partial \varphi^{l,r}(x_0(t), t)}{\partial x_0}  \delta(t) -\epsilon^{l,r}(x_0(t),t) \right) \frac{\partial Q_0^{l,r}}{\partial \xi_\beta} (\eta,t) d\eta ds \bigg),
\end{multline}
where $ \displaystyle z^{l,r}(\xi,t)= \left( \frac{\partial Q_0^{l,r}}{\partial \xi} (0,t) \right)^{-1} \frac{\partial Q_0^{l,r}}{\partial \xi} (\xi,t)$.

The equations defining the functions \( q_{1}^{l,r}\left(\xi_{\beta}, t\right) \) are derived in a similar way to \eqref{q0equation}, by equating the coefficients at the small parameter \( \mu^2 \), with the additional boundary conditions \( q_{1}^{l,r}(0, t) = 0 \) and \( q_{1}^{l,r}(\mp \infty, t) = 0 \).

The lower solution \(\alpha(x, t, \mu)\) is constructed in an analogous manner:

\begin{align*}
\alpha^{l}(x, t, \mu) &= U_{n+1}^{l}\Big|_{\xi_{\alpha}} - \mu^{n+1}\left[\epsilon^{l}(x,t) + q_{0}^{l}(\xi_{\alpha}, t) + \mu q_{1}^{l}(\xi_{\alpha}, t)\right], \\
&\qquad\qquad\qquad\qquad \text{for } 0 \leq x \leq x_{\alpha}, \, \xi_{\alpha} \leq 0, \, t \in \mathbb{R}_+, \\
\alpha^{r}(x, t, \mu) &= U_{n+1}^{r}\Big|_{\xi_{\alpha}} - \mu^{n+1}\left[\epsilon^{r}(x,t) + q_{0}^{r}(\xi_{\alpha}, t) + \mu q_{1}^{r}(\xi_{\alpha}, t)\right], \\
&\qquad\qquad\qquad\qquad \text{for } x_{\alpha} \leq x \leq 1, \, \xi_{\alpha} \geq 0, \, t \in \mathbb{R}_+.
\end{align*}

Here, \(\xi_{\alpha}\) is defined as:
\[
\xi_{\alpha} = \frac{x - x_{\alpha}(t)}{\mu},
\]
where \(x_{\alpha}(t) = X_{n+1}(t) + \mu^{n+1} \delta(t)\).

Now, we need to prove that these functions satisfy Conditions C1-C4.

C1. Verification of the ordering condition:

We analyze the difference between $\beta$ and $\alpha$ over three intervals:

\[
\beta(x, t, \mu) - \alpha(x, t, \mu) = 
\begin{cases}
\beta^{l}(x, t, \mu) - \alpha^{l}(x, t, \mu), & 0 \leq x < x_{\beta}(t), \\
\beta^{r}(x, t, \mu) - \alpha^{l}(x, t, \mu), & x_{\beta}(t) \leq x \leq x_{\alpha}(t), \\
\beta^{r}(x, t, \mu) - \alpha^{r}(x, t, \mu), & x_{\alpha}(t) < x \leq 1.
\end{cases}
\]

On the interval $[x_{\beta}(t), x_{\alpha}(t)]$, we obtain:

\begin{align*}
\beta^{r} - \alpha^{l} &= \sum_{i=0}^{n-1} \mu^{i}\left(\bar{U}_{i}^{r}(x) + Q_{i}^{r}\left(\xi_{\beta}, t\right)\right) 
 - \sum_{i=0}^{n-1} \mu^{i}\left(\bar{U}_{i}^{l}(x) + Q_{i}^{l}\left(\xi_{\alpha}, t\right)\right) + O\left(\mu^{n}\right) \\
&= \frac{\partial Q_0}{\partial \xi}\left(0, t\right)\left(\xi_{\beta} - \xi_{\alpha}\right) + O\left(\mu^{n}\right) = 2 \mu^{n-1} \delta(t)\frac{\partial Q_0}{\partial \xi}\left(0, t\right) + O\left(\mu^{n}\right), \text{for } t\in \mathbb{R}_+.
\end{align*}

Given Assumption \ref{Cond2} and equations \eqref{QL0Int} and \eqref{QR0Int}, we have $\frac{\partial Q_0}{\partial \xi}\left(0, t\right) > 0$. Thus, $\beta^{r} - \alpha^{l} > 0$ on the interval $\left[x_{\beta}(t), x_{\alpha}(t)\right]$, provided $\delta(t) > 0$ for $t \in \mathbb{R}_+$ and sufficiently small $\mu$.

Similar analysis for the intervals $[0, x_{\beta}(t)]$ and $[x_{\alpha}(t), 1]$ completes the verification of the ordering condition.

C2. Verification of the differential inequalities:

By substituting $\beta^{l,r}$ and $\alpha^{l,r}$ into the respective operators $L[\beta]$ and $L[\alpha]$, and considering that $\epsilon^{l,r}$, $q_{0}^{l,r}$, and $q_{1}^{l,r}$ satisfy their respective equations, we obtain the following expressions:
\[
L[\beta]=-\mu^{n+1}R+O\left(\mu^{n+2}\right)<0,\quad L[\alpha]=\mu^{n+1}R+O\left(\mu^{n+2}\right)>0,
\]
where $R > 0$ is the constant on the right-hand side of the equation corresponding to $\epsilon^{l,r}$.

C3. Verification of the boundary conditions:

The boundary conditions hold when $R^{l}$ and $R^{r}$ in the boundary conditions for $\epsilon^{l,r}$ are chosen sufficiently large.

C4. Verification of the continuity conditions:
From the explicit solution for $  q_0^{l,r}(0,t)$ \eqref{q0explisit}, we obtain
\begin{align*}
\begin{split}
& \frac{\partial q_0^l}{\partial \xi}(0,t) - \frac{\partial q_{0}^r}{\partial \xi}(0,t)  =\frac{1}{2}( \varphi^r(x_0(t),t) -\varphi^l(x_0(t),t))   \\ &\qquad\qquad
\cdot \left(  \epsilon^l(x_0(t),t)+\epsilon^r(x_0(t),t)- \delta(t) \left( \frac{\partial \varphi^{l}(x_0(t), t)}{\partial x_0}  +\frac{\partial \varphi^{r}(x_0(t), t)}{\partial x_0} \right)  \right) .
\end{split}
\end{align*}

Choosing $\delta(t)$ as the solution of:
\[
\delta(t)=\frac{\epsilon^l(x_0(t),t)+\epsilon^r(x_0(t),t)}{ \left( \frac{\partial \varphi^{l}(x_0(t), t)}{\partial x_0}  +\frac{\partial \varphi^{r}(x_0(t), t)}{\partial x_0} \right) }-\sigma,
\]
where $\sigma\geq0$, we obtain:
\[
\left.\left(\frac{\partial\beta^{l}}{\partial x}-\frac{\partial\beta^{r}}{\partial x}\right)\right|_{x=x_{\beta}(t)}=\frac{\sigma}{2}\mu^{n}\left(\varphi^{r}(x_0(t),t)-\varphi^{l}(x_0(t),t)\right)+O\left(\mu^{n+1}\right).
\]

By Assumption 1, $\varphi^{r}(x_0(t),t)-\varphi^{l}(x_0(t),t)>0$, so the right-hand side is non-negative. The same reasoning applies to the lower solution.

Thus, we have verified that the constructed upper and lower solutions satisfy all required Conditions C1-C4.
 From the existence of these upper and lower solutions, it follows that a unique solution exists. This completes the proof of the Theorem \ref{existenceTheorem}.
\end{proof}

\begin{proof}[Proof of Theorem \ref{asymptEstimates} (Asymptotic estimates)]

\textbf{ 1. Proof of estimate \eqref{estim1}: }

From the construction of the upper and lower solutions and the comparison principle \cite{Hess1991-ha}, we have:
\[
\alpha(x, t, \mu) \leq u(x, t, \mu) \leq \beta(x, t, \mu), \quad (x, t) \in \overline{\mathcal{D}}.
\]

We also know that $\beta(x, t, \mu)-\alpha(x, t, \mu)=O\left(\mu^{n-1}\right)$. Therefore,
\[
u(x, t, \mu)=\alpha(x, t, \mu)+O\left(\mu^{n-1}\right)=U_{n-2}(x, t, \mu)+O\left(\mu^{n-1}\right).
\]

Replacing $n-2$ by $n$, we obtain:
\[
\left|u(x, t, \mu)-U_{n}(x, t, \mu)\right| \leq C_{n} \mu^{n+1}.
\]

This proves estimate \eqref{estim1}.

\textbf{2. Proof of estimate \eqref{estim3}:}

 We demonstrate that the estimate \eqref{estim3} is valid by analyzing the difference \( z_n(x, t, \mu) = u(x, t, \mu) - U_n(x, t, \mu) \). The function \( z_n(x, t, \mu) \) satisfies the differential equation
\[
\mu \left( \frac{\partial^{2}z_n}{\partial x^{2}} - \frac{\partial z_n}{\partial t} \right) - \left(  U_n \frac{\partial U_n}{\partial x} -  u \frac{\partial u}{\partial x} \right) = \mu^{n+1}\psi(x, t, \mu),
\]
for \( (x, t) \in \overline{\Omega} \) with zero boundary conditions, where \( \lvert \psi(x, t, \mu) \rvert \leq c_1 \).

Defining \( r(x, t, \mu) = \mu^n \psi(x, t, \mu) \) and using a Green's function for the parabolic operator on the left-hand side, we represent \( z_n \) as follows:
\begin{multline*}
z_n(x, t, \mu) = \int_{0}^{1} G(x,  t, \zeta,  t_0) z_n(\zeta,  t_0) d\zeta \\
+\frac{\partial}{\partial x} \left( \int_{ t_{0}}^{ t}d\tau\int_{0}^{1}G_{x}(x,  t, \zeta, \tau)\frac{1}{\mu}\int_{u(\zeta,\tau,\mu)}^{U_{n}(\zeta,\tau,\mu)}( s) ds d\zeta \right).
\end{multline*}

Further analysis, following the approach in \cite{ChaZha2022}, of this representation shows that the derivative $\frac{\partial z_n}{\partial x}(x,  t, \mu)$ is bounded by $\mathcal{O}(\mu^n)$ for $ (x, t) \in \overline{\Omega} $. This completes the proof of estimation  \eqref{estim3}.

\textbf{3.Proof of estimates for outer regions:}

The Assumption \ref{Cond1} allows us to derive the interval of the location of inner transition layer $[\hat{x}^{l}(t,\mu),\hat{x}^{r}(t,\mu)]$.

From the  Assumption \ref{Cond1}, we deduce:
\begin{equation*}
|Q_{0}^{l,r}(0,t)| = |\varphi(x_{0}(t)) - \varphi^{l,r}(x_{0}(t))| = \frac{1}{2}(\varphi^{r}(x_{0}(t)) - \varphi^{l}(x_{0}(t))) \geq \frac{1}{2}(u^r - u^l) >\mu^2,
\end{equation*}
and $|Q_{0}^{l,r}(\xi,t)| \to 0$ as $\xi \to \mp \infty$.
Given that $0<\mu < 1$ is a fixed number, $|Q_{0}^{l}(\xi,t)|$ is increasing function and $|Q_{0}^{r}(\xi,t)|$ is decreasing function, and $\xi = (x - x_{t.p.}(t)) / \mu$, there exist $\hat{x}^{l,r}(t,\mu)$ such that:
on the interval $[0, \hat{x}^{l}(t,\mu)]$,  $|Q_{0}^{l}(\xi, t)| \leq \mu^2$ for every $t$;
on the interval $[\hat{x}^{r}(t,\mu), 1]$,  $|Q_{0}^{r}(\xi, t)| \leq \mu^2$ for every $t$.

Furthermore, at the points $\hat{x}^{l,r}(t,\mu)$:
\begin{equation*}\label{Qlr-mu}
|Q_{0}^{l}(\xi(\hat{x}^{l}(t,\mu)), t)| = \mu^2, \quad |Q_{0}^{r}(\xi(\hat{x}^{r}(t,\mu)), t)| = \mu^2.
\end{equation*}

For the right side, starting from the explicit form for \( Q_0^l(\xi_0,t) \):
\[
Q_0^l(\xi_0,t) =  \frac{K_1^l}{1-K_2^l \exp(\varphi^l(x_0(t))\xi_0)} ,
\]
where \( \xi_0 = \frac{x - x_0(t)}{\mu} \), $ K_1^l= -2\varphi^l(x_0(t)) $, $ K_2^l=\frac{\varphi^r(x_0(t))+3\varphi^l(x_0(t))}{\varphi^r(x_0(t))-\varphi^l(x_0(t))} $.

At the point \( \hat{x}^l(t,\mu) \), we obtain:
\[
\mu^2 = \frac{K_1^l}{1-K_2^l \exp(\varphi^l(x_0(t)) \frac{\hat{x}^l(t,\mu) - x_0(t)}{\mu})}.
\]

Solving for \( \hat{x}^l(t,\mu) \), we find:
\begin{equation}\label{hatxl}
\hat{x}^l(t,\mu) = x_0(t) - \frac{\mu}{\varphi^l(x_0(t))} \ln \left(\frac{K_2^l \mu^2}{-K_1^l+\mu^2}\right).
\end{equation}

Similarly, we obtain the estimate for the right side:

\begin{equation}\label{hatxr}
\hat{x}^r(t,\mu) = x_0(t) - \frac{\mu}{\varphi^r(x_0(t))} \ln \left(\frac{K_2^r \mu^2}{-K_1^r+\mu^2}\right),
\end{equation}
where  $ K_1^r= -2\varphi^r(x_0(t)) $, $ K_2^r=\frac{\varphi^l(x_0(t))+3\varphi^r(x_0(t))}{\varphi^l(x_0(t))-\varphi^r(x_0(t))}. $

Now, focusing on the left side of the inner layer, the proof for the right side is similar. We need to prove that for  \( x \in \left[0, \hat{x}^l(t,\mu)\right] \), where \( \hat{x}^r(t,\mu) -\hat{x}^l(t,\mu) = \mathcal{O}(\mu |\ln\mu|) \):
\[
|u(x,t)-\varphi^l(x)| \leq C\mu, \quad \left|\frac{\partial u}{\partial x}(x,t) - \frac{d\varphi^l}{dx}(x)\right| \leq C\mu.
\]

From the construction of \( U_0(x,t) \), we know that for \( x < x_0(t) \):
\[
U_0(x,t) = \varphi^l(x) + Q_0^l(\xi_0,t),
\]
where \( \xi_0 = \frac{x-x_0(t)}{\mu} \). 

Now, using estimate \eqref{estim1} with $n=0$, we have for \( x \in \left[0,\hat{x}^l(t,\mu)\right] \):
\begin{align*}
|u(x,t)-\varphi^l(x)| &\leq |u(x,t) - U_0(x,t,\mu)| + |U_0(x,t,\mu) - \varphi^l(x)| \\
&\leq \mathcal{O}(\mu) + |Q_0^l(\xi_0,t)| \\
&\leq \mathcal{O}(\mu) +\mu^2 \\
&\leq C\mu,
\end{align*}
where constant \( C \) is not depend from $x,t, \mu $.

For \( x \leq \hat{x}^l(t,\mu) \),  we have $ \xi_0 \leq \frac{\hat{x}^l(t,\mu)-x_0(t)}{\mu} = - \frac{1}{\varphi^l(x_0(t))} \ln \left(\frac{K_2^l \mu^2}{-K_1^l+\mu^2}\right) $.

Since $\frac{\partial Q_0^l}{\partial \xi_0}(\xi_0,t)$ is an exponentially increasing function, its maximum on the interval \( x \in \left[0,\hat{x}^l(t,\mu)\right] \) will be at  \( x = \hat{x}^l(t,\mu)  \). Thus, from the explicit equation for $Q_0^l$:
\[
\frac{\partial Q_0^l}{\partial x}(\xi_0,t)  = \frac{1}{\mu} K_3 e^{\varphi^l(x_0) \xi_0} (Q_0^l)^2 \leq K_3  \left(\frac{K_2^l \mu^2}{-K_1^l+\mu^2}\right)^{-1} \mu^{3},
\]
where $K_3=\frac{\varphi^r(x_0(t))+3\varphi^l(x_0(t))}{2(\varphi^l(x_0(t))-\varphi^r(x_0(t)))}$.

Therefore, we obtain for \( x \in \left[0,\hat{x}^l(t,\mu)\right] \):
\begin{align*}
\left|\frac{\partial u}{\partial x}(x,t) - \frac{d\varphi^l}{dx}(x)\right| &\leq \left|\frac{\partial u}{\partial x}(x,t) - \frac{\partial U_0}{\partial x}(x,t,\mu)\right| + \left|\frac{\partial U_0}{\partial x}(x,t,\mu) - \frac{d\varphi^l}{dx}(x)\right| \\
&\leq \mathcal{O}(\mu) + \frac{\partial Q_0^l}{\partial x}(\xi_0,t) \\
&\leq \mathcal{O}(\mu) + K_3  \left(\frac{-K_1^l+\mu^2}{K_2^l }\right) \mu \\
&\leq C\mu,
\end{align*}
where constant \( C \) is not depend from $x,t, \mu $.

\textbf{4. Proof of estimate \eqref{estim2}:}

We begin by analyzing \( u(x, t) - \varphi(x, t) \) outside the transition layer \( [\hat{x}^l(t, \mu), \hat{x}^r(t, \mu)] \), where the function $\varphi(x,t) $ is defined as equation \eqref{barUatTP}:  \( \varphi(x,t) = \frac{1}{2} \left( \varphi^l(x,t) + \varphi^r(x,t) \right) \), while $\hat{x}^l(t, \mu)$ and $ \hat{x}^r(t, \mu)$ are defined in \eqref{hatxl} and \eqref{hatxr}, respectively.

Case 1: \( x \leq \hat{x}^l(t, \mu) \). According to the estimate \eqref{leftestimatevarphi}, we have
\[
|u(x, t) - \varphi^l(x, t)| \leq C \mu, \quad \text{for } x \leq \hat{x}^l(t, \mu),
\]
where \( C \) is a constant independent of \( \mu \), \( x \), and \( t \).

On the other hand, let us consider the following quantity:
\[
\delta(x,t) = \varphi^r(x, t) - \varphi^l(x, t).
\]

According to Assumption \ref{Cond1}, the function $\delta(x,t)$ satisfies:
$$ \delta(x,t) \geq 2\mu^\theta > 0, \  \text{where} \ \theta \in (0,1), \quad \forall x \in [0,1], \forall t \in \mathbb{R}_+. $$

Therefore,  for \( \mu^{(1-\theta)} < \frac{1}{C}\), we have:
\[
u(x, t) - \varphi(x, t) = [u(x, t) - \varphi^l(x, t)] - \frac{1}{2} \delta(x,t) \leq C \mu - \mu^\theta < 0, \ \text{for} \ x \leq \hat{x}^l(t, \mu) .
\]

Case 2:  \( x \geq \hat{x}^r(t, \mu) \). The estimate \eqref{rightestimatevarphi} gives:
\[
|u(x, t) - \varphi^r(x, t)| \leq C \mu, \quad \text{for } x \geq \hat{x}^r(t, \mu).
\]

Similarly, for  \( \mu^{(1-\theta)} < \frac{1}{C}\), we obtain:
\[
u(x, t) - \varphi(x, t) = [u(x, t) - \varphi^r(x, t)] + \frac{1}{2} \delta(x,t) \geq -C \mu + \mu^\theta > 0, \ \text{for} \  x \geq \hat{x}^r(t, \mu) .
\]

Since both \( u(x, t) \) and \( \varphi(x, t) \) are continuous in \( x \) on \( [0,1] \), \( u(x, t) - \varphi(x, t) \) is also continuous in \( x \) on \( [0,1] \). Furthermore, because \( u(x, t) - \varphi(x, t) \) is continuous on \( [\hat{x}^l(t, \mu), \hat{x}^r(t, \mu)] \) and changes sign from negative to positive, there exists a point \( x_{t.p.}(t) \in (\hat{x}^l(t, \mu), \hat{x}^r(t, \mu)) \) such that:
\[
u(x_{t.p.}(t), t) - \varphi(x_{t.p.}(t), t) = 0.
\]

Thus, for \( 0 < \mu \ll 1 \) (since \( x_0 \)  lies within the interval \( (\hat{x}^l(t, \mu), \hat{x}^r(t, \mu)) \) and $\Delta x =\hat{x}^r(t, \mu)-\hat{x}^l(t, \mu) = \mathcal{O}(\mu |\ln\mu|)$), we have:
\begin{equation*} \label{xtp_in_transition}
|x_0 - x_{t.p}| < C \mu |\ln \mu |, \quad \forall t \in \mathbb{R}_+,
\end{equation*}
where \( C \) is a constant independent of \( \mu \), \( x \), and \( t \).

This completes the proof of Theorem \ref{asymptEstimates}.
\end{proof}

\end{document}